\documentclass[11pt]{amsart}

\usepackage[margin=1in]{geometry}
\usepackage{amsmath,amsthm,amsfonts,amssymb,amscd}
\usepackage{lastpage}
\usepackage{fancyhdr}
\usepackage{mathrsfs}
\usepackage{xcolor}
\usepackage{graphicx}
\usepackage{listings}
\usepackage{hyperref}
\usepackage{enumitem}
\usepackage{relsize}
\usepackage{tikz-cd}
\usepackage{mdwlist}
\usepackage{multicol}
\usepackage{stmaryrd}
\usepackage{float}
\usepackage{adjustbox}
\usepackage{tikz}
\usepackage[pagewise]{lineno}
\usetikzlibrary{matrix, calc, arrows}

\hypersetup{
    colorlinks=true, %set true if you want colored links
    linktoc=all,     %set to all if you want both sections and subsections linked
    linkcolor=blue,  %choose some color if you want links to stand out
}

\setlist[enumerate]{label=(\alph*)}
\usetikzlibrary{graphs,decorations.pathmorphing,decorations.markings}
\tikzcdset{scale cd/.style={every label/.append style={scale=#1},
    cells={nodes={scale=#1}}}}

\newcommand{\Z}{\mathbb{Z}}

\newcommand{\Hom}{\operatorname{Hom}}

\newcommand{\tr}{\operatorname{tr}}

\newcommand{\inv}{^{-1}}

\newcommand{\End}{\operatorname{End}}

\newcommand{\Res}{\operatorname{Res}}
\newcommand{\Ind}{\operatorname{Ind}}

\newcommand{\id}{\operatorname{id}}

\newcommand{\stab}{\operatorname{stab}}
\newcommand{\calA}{\mathcal{A}}
\newcommand{\calB}{\mathcal{B}}
\newcommand{\calC}{\mathcal{C}}

\newcommand{\calF}{\mathcal{F}}

\newcommand{\calO}{\mathcal{O}}
\newcommand{\calP}{\mathcal{P}}

\newcommand{\calX}{\mathcal{X}}

\newcommand{\triv}{\mathbf{triv}}

\newcommand{\catmod}{\mathbf{mod}}
\newcommand{\Br}{\operatorname{Br}}
\newcommand{\br}{\operatorname{br}}
\newcommand{\Ch}{\operatorname{Ch}}

\newtheorem{theorem}{Theorem}[section]
\newtheorem{lemma}[theorem]{Lemma}
\newtheorem{prop}[theorem]{Proposition}
\newtheorem{corollary}[theorem]{Corollary}

\newtheorem*{theorem*}{Theorem}

\theoremstyle{remark}
\newtheorem{remark}[theorem]{Remark}

\theoremstyle{definition}
\newtheorem{definition}[theorem]{Definition}

    \title{Brauer pairs for splendid Rickard equivalences}

    \author{Jadyn V. Breland}
    \address{Department of Mathematics, University of California Santa Cruz, 
    Santa Cruz, CA 95064}
    \email{jbreland@ucsc.edu}
    
    \author{Sam K. Miller}
    \address{Department of Mathematics, University of California Santa Cruz, 
    Santa Cruz, CA 95064}
    \email{sakmille@ucsc.edu}
     %required
     %optional

\begin{document}
	
	\maketitle
     \begin{abstract}
        We define the notion of a Brauer pair of a chain complex, extending the notion of a Brauer pair of a $p$-permutation module introduced by Boltje and Perepelitsky. In fact, the Brauer pairs of a splendid Rickard equivalence $C$ coincide with the set of Brauer pairs of the corresponding $p$-permutation equivalence $\Lambda(C)$ induced by $C$. As a result, we derive structural results for splendid Rickard equivalences that correspond to known structural properties for $p$-permutation equivalences. In particular, we show splendid Rickard equivalences induce local splendid Rickard equivalences between normalizer block algebras as well as centralizer block algebras.  
    \end{abstract}

    \section{Introduction}

    Brou\'e's abelian defect group conjecture, one of the hallmark conjectures of modular representation theory, predicts an equivalence between a block with abelian defect group and a corresponding ``local'' block associated to the normalizer of the defect group. Moreover, these equivalences should induce ``local equivalences;'' for these block algebras, there are corresponding pairs of blocks associated to centralizers which are also equivalent. 
    
    Historically, these equivalences were first predicted by Brou\'e to exist either as \textit{isotypies}, a bouquet of character-theoretic equivalences called \textit{perfect isometries} for pairs of local subgroups, or as derived equivalences, which then induced a perfect isometry. However, Brou\'e suspected that isotypies were the shadow of a categorical equivalence from which all of the components of the isotypy are induced. Rickard refined this conjecture in \cite{Ri96} by introducing ``splendid'' derived equivalences, which are now often called splendid Rickard equivalences. These equivalences are chain complexes of $p$-permutation bimodules which have ``diagonal'' vertices, and induce not only derived equivalences, but equivalences between homotopy categories of $p$-permutation modules for the block algebras as well. In fact, recent work of Balmer and Gallauer in \cite{BG23a} demonstrates that these homotopy categories of $p$-permutation modules admit the corresponding derived module categories as Verdier quotients. Rickard's key insight was that $p$-permutation is the correct property to require for the existence of equivalences between local subgroups. In this case, the Brauer construction, a common construction when working with $p$-permutation modules, induces the local equivalences. Such complexes were later extended to the language of (almost) source algebras by Linckelmann in \cite{Li98}.
		
	In \cite{BoXu07}, Boltje and Xu introduced $p$-permutation equivalences, a type of equivalence for blocks which exists on the representation ring level. Such equivalences live between isotypies and splendid Rickard equivalences, in the sense that in order to induce an isotypy from a splendid Rickard equivalence, one first constructs a corresponding $p$-permutation equivalence by taking the Lefschetz invariant of the chain complex, i.e. taking an alternating sum of terms. Much of the theory of splendid Rickard equivalences and $p$-permutation equivalences should be expected to coincide - for instance, to induce local equivalences between centralizer block algebras for both types of equivalences, one applies the Brauer construction in the same manner. However, a splendid Rickard equivalence has additional structure that a $p$-permutation equivalence does not, those of differentials and grading. 
 
    The earliest notions of $p$-permutation equivalence in \cite{BoXu07} made restrictive assumptions, namely that the blocks have a common defect group and that certain block fusion systems are equivalent. Boltje and Perepelistky in \cite{BoPe20} gave a wider definition for $p$-permutation equivalences, allowing for direct sums of blocks, removing the assumption of a common defect group, and allowing the vertices of any indecomposable summand to be ``twisted diagonal.'' Under these looser assumptions, it is initially unclear how one would obtain local equivalences, or if the stricter structural properties of a $p$-permutation equivalence would hold in the case of block equivalences. The machinery of Brauer pairs was necessary to resolve these questions, and in fact, many of the structural invariants, such as blocks having isomorphic fusion systems, continue to hold. By applying the Brauer construction to a $p$-permutation equivalence at certain Brauer pairs, one obtains a local $p$-permutation equivalence, leading to the notion of a Brauer pair of a $p$-permutation equivalence. If $A$ and $B$ are direct summands of group algebras $kG$ and $kH$ and $\gamma$ is a $p$-permutation equivalence for $A$ and $B$, we say the pair $(\Delta(P,\phi,Q), e\otimes f^*)$, where $\Delta(P, \phi, Q) \leq G\times H$ is a twisted diagonal subgroup, $e$ a block of $kC_G(P)$, and $f$ a block of $kC_H(Q)$, is a \textit{$\gamma$-Brauer pair} if $e\gamma(\Delta(P,\phi, Q))f \neq 0$. 
    
    Boltje and Perepelitsky developed this theory further, deducing the structure of the poset of Brauer pairs for $p$-permutation equivalences. Moreover, they showed that the local equivalences induced extend beyond blocks of centralizer group algebras - they may be extended to stronger equivalences for normalizer blocks as well. A key observation was that applying the Brauer construction to a $p$-permutation equivalence does not immediately lead to a centralizer $p$-permutation equivalence, but rather a module which is ``invertible'' under an extended tensor product. This suggests that there is more structure which is missed by only considering local centralizer $p$-permutation equivalences. 

    The goal of this paper is adapting much of the theory in \cite{BoPe20} to splendid Rickard equivalences. We adopt a similarly flexible definition for splendid Rickard equivalences, allowing for twisted diagonal vertices and for the equivalence to be between sums of blocks. We then introduce the notion of a Brauer pair for a splendid Rickard equivalence. Given a splendid Rickard equivalence $C$ for direct summands $A$ and $B$ of group algebras $kG$ and $kH$, respectively, we say $(\Delta( P, \phi, Q), e\otimes f^*)$ is a $C$-Brauer pair if $eC(\Delta(P,\phi, Q))f$ is a noncontractible chain complex. Our main result is the following, which is Corollary \ref{mainresult}: 

	\begin{theorem}
		Let $C$ be a splendid Rickard equivalence between $A$ and $B$ and let $\gamma:= \Lambda(C)$ be its corresponding $p$-permutation equivalence. Then $\mathcal{BP}(C) = \mathcal{BP}(\gamma)$. In particular, 
  \begin{enumerate}
    \item $\mathcal{BP}(C)$ is a $G\times H$-stable ideal in the poset of $A\otimes B^*$-Brauer pairs. 
      \item If $A$ and $B$ are blocks, then any two maximal $C$-Brauer pairs are $G\times H$-conjugate.
      \item The following are equivalent for $(\Delta(P,\phi,Q),e\otimes f^*) \in \calB\calP(C)$:
      \begin{enumerate}[label=(\roman*)]
          \item $(\Delta(P,\phi,Q),e\otimes f^*)$ is a maximal $C$-Brauer pair.
          \item $(P,e)$ is a maximal $A$-Brauer pair.
          \item $(Q,f)$ is a maximal $B$-Brauer pair.
      \end{enumerate}
      \item Suppose $A$ and $B$ are blocks and $(\Delta(D,\phi,E),e_D\otimes f_E^*)$ is a maximal $C$-Brauer pair. Let $\mathcal{A}$ denote the fusion system of $A$ associated to $(D,e_D)$ and let $\mathcal{B}$ denote the fusion system of $B$ associated to $(E,f_E)$.  Then $\phi:E \to D$ is an isomorphism of the fusion systems $\mathcal{B}$ and $\mathcal{A}$.
    \end{enumerate}
	\end{theorem}
 
    These results were deduced by Boltje and Perepelitsky in \cite{BoPe20} in the context of Brauer pairs for $p$-permutation equivalences; our main contribution is showing that the notion of Brauer pairs agree. Part (d) recovers a classical result of Puig; see for instance \cite{Pu99}. We deduce that splendid Rickard equivalences induce equivalences at the normalizer level as well. This is Corollary \ref{normalizerequivalence}.

    \begin{theorem}
        Let $C$ be a splendid Rickard equivalence for blocks $A\subseteq kG$ and $B\subseteq kH$ and let $(\Delta(P,\phi,Q), e\otimes f^*)$ be a $C$-Brauer pair. Set $I=N_G(P,e)$ and $J=N_H(Q,f)$ and let $\hat{e}$ and $\hat{f}$ denote the unique block idempotents of $k[N_G(P)]$ and $k[N_H(Q)]$ which cover $e$ and $f$, respectively. Then, 
		\[\hat{e}k[N_G(P)]e \otimes_{kI} \Ind^{I\times J}_{N_{I \times J}(\Delta(P,\phi, Q))}\big(eC(\Delta(P, \phi, Q))f\big) \otimes_{kJ} fk[N_H(Q)]\hat{f} \]
		is a splendid Rickard equivalence for block algebras $k[N_G(P)]\hat{e}$ and $k[N_H(Q)]\hat{f}$. 	
    \end{theorem}

    Finally, we deduce that the vertices of the bimodules which appear in an indecomposable splendid Rickard equivalence $C$ are governed by any single maximal Brauer pair for $C$. This is Corollary \ref{bpswhichappear}.

    \begin{theorem}
        Let $C$ be an indecomposable splendid Rickard equivalence for block algebras $A\subseteq kG$ and $B\subseteq kH$ and let $(\Delta(D,\phi,E),e\otimes f^*)$ be a maximal $C$-Brauer pair. Every trivial source bimodule occurring in $C$ has a maximal Brauer pair contained in $(\Delta(D,\phi,E),e\otimes f^*)$. In particular, the set consisting of all Brauer pairs for every trivial source bimodule appearing in $C$ and the set of $C$-Brauer pairs coincide. 
    \end{theorem}

    The paper is organized as follows. Section 2 introduces some preliminary definitions involving direct products of groups. Section 3 reviews some prerequisites on $p$-permutation modules and the Brauer construction. Section 4 is a review of block theory, Brauer pairs, and the fusion system of a block. Section 5 establishes the definitions of $p$-permutation equivalences and splendid Rickard equivalences we will use for the paper, and proves a few basic results for splendid Rickard equivalences which hold for more restrictive definitions. Section 6 reviews preliminaries on extended tensor products with a focus on extended tensor products of $p$-permutation bimodules. Section 7 recalls the definition of a Brauer pair for a $p$-permutation module and a $p$-permutation equivalence as defined by Perepelitsky and Boltje. Section 8 states and proves the main results of the paper regarding the notion of a Brauer pair for a splendid Rickard equivalence. Finally, Section 9 proves as corollaries a few structural results about the modules which appear in splendid Rickard equivalences, which are analogous to results for $p$-permutation equivalences.

    \textbf{Notation and conventions:} For this paper, $G$ and $H$ are finite groups, $k$ is an algebraically closed field of prime characteristic $p$, and $\calO$ is a complete discrete valuation ring of characteristic 0 with residue field $k$. The symbol $R$ may be used to denote either $\calO$ or $k$. All modules are assumed to be finitely generated, and all chain complexes are assumed to be bounded. Given two $RG$-modules $M$ and $N$, the tensor product $M\otimes_k N$ is an $RG$-module with $G$-action $g(m\otimes n) = gm\otimes gn$. We denote by $M^*:=\mathrm{Hom}_R(M,R)$ the $R$-dual of $M$, which is an $RG$-module via $G$-action $g\cdot f(m) = f(g\inv m)$. Similarly, if $K$ is a finite group, given a $(RG,RH)$-bimodule $M$ and a $(RH,RK)$-bimodule $N$, the tensor product $M\otimes_{RH} N$ is a $(RG,RK)$-bimodule. We denote by $M^*$ the $R$-dual of $M$ which is an $(RH,RG)$-bimodule under the action $h\cdot f(m) \cdot g = f(gmh)$. We write $s_p(G)$ to denote the set of $p$-subgroups of $G$. We refer the reader to \cite[Sections 1.17]{Li181} for the constructions of tensor products and internal homs of chain complexes.

	\section{Subgroups of direct products of groups}
		
	We begin by recalling some constructions related to subgroups of direct products of groups. We refer the reader to \cite[Part I, Chapter 2]{Bo10} for more details. 
	
	\begin{definition}
		Let $G,H,K$ be finite groups, let $X \leq G\times H$ and $Y \leq H\times K$. 
		\begin{enumerate}
			\item Denote by $p_1: G\times H \to G$ and $p_2: G\times H \to H$ the canonical projections. Setting \[k_1(X) := \{g \mid (g,1) \in X\} \text{ and } k_2(X):= \{h \mid (1, h) \in H\},\] we obtain normal subgroups $k_i(X) \trianglelefteq p_i(X)$ and canonical isomorphisms $X/(k_1(X) \times k_2(X)) \xrightarrow{\sim} p_i(X)/k_i(X)$ for $i \in \{1,2\}$ induced by the projection maps $p_i$. The resulting isomorphism $\nu_X: p_2(X)/k_2(X) \xrightarrow{\sim} p_1(X)/k_1(X)$ satisfies $\nu_X(hk_2(X)) = gk_1(X)$ if and only if $(g,h) \in X$, where $(g,h) \in p_1(X)\times p_2(X)$. 
			\item If $\phi: Q \xrightarrow{\sim} P$ is an isomorphism between subgroups $Q \leq H$ and $P \leq G$, then \[\Delta(P, \phi, Q) := \{(\phi(q),q) \mid q\in Q\} \] is a subgroup of $G\times H$. Subgroups of this form will be called \textit{twisted diagonal subgroups} of $G\times H$. For $P \leq G$, we set $\Delta(P):= \Delta(P, \id, P)\leq G\times G$. It is easy to verify that $X \leq G\times H$ is twisted diagonal if and only if $|k_1(X)| = 1$ and $|k_2(X)| = 1$, and that for $(g,h) \in G\times H$, one has \[{}^{(g,h)} \Delta(P,\phi, Q) = \Delta({}^gP, c_g\phi c_{h\inv}, {}^h Q). \]
            Moreover, the set of twisted diagonal subgroups of $G\times H$ is closed under taking subgroups. 
			\item The subgroup $X^\circ := \{(h,g) \in H\times G \mid (g,h) \in X \} $ is called the \textit{subgroup opposite to $X$}. 
			\item The \textit{composition} of $X$ and $Y$ is defined as \[X\ast Y := \{(g,k) \in G\times K \mid \exists h\in H: (g,h)\in X \text{ and } (h,k)\in Y \}.  \] It is a routine verification that composition is associative. If $\Delta(P,\phi, Q) \leq G\times H$ and $\Delta (Q, \psi, R) \leq H\times K$ are twisted diagonal subgroups, then \[\Delta(P,\phi, Q) \ast \Delta(Q,\psi, R) = \Delta(P, \phi\circ\psi, R). \] Note that $(X\ast Y)^\circ = Y^\circ \ast X^\circ$.
		\end{enumerate}
	\end{definition}
	
	The following lemma consists of routine verifications.
	
	\begin{lemma}\label{lem:composing subgroups}
		Let $G,H,K$ be finite groups, let $X \leq G\times H$ and $Y \leq H\times K$. 
		\begin{enumerate}
			\item We have $X \ast X^\circ= \Delta(p_1(X))\cdot (k_1(X)\times \{1\}) = \Delta(p_1(X))\cdot (\{1\}\times k_1(X)) = \Delta(p_1(X))\cdot (k_1(X)\times k_1(X))$.
			\item We have $X\ast X^\circ \ast X = X.$
			\item For any $g \in G, h\in H, k\in K$, we have ${}^{(g,h)}X \ast {}^{(h,k)}Y = {}^{(g,k)}(X\ast Y)$.
		\end{enumerate}
	\end{lemma}
	
	\begin{definition}
		Let $G$ and $H$ be finite groups, let $\Delta(P, \phi, Q)$ be a twisted diagonal subgroup of $G\times H$, and let $S \leq N_G(P)$ and $T\leq N_H(Q)$. We denote by $N_{(S,\phi, T)}$ the subgroup of $S$ consisting of all elements $g \in S$ such that there exists an element $h \in T$ satisfying $c_g \phi c_h = \phi$ as functions from $Q$ to $P$. Note that if $S$ contains $C_G(P)$, then also $N_{(S, \phi, T)}$ also contains $C_G(P)$. Moreover, if $P \leq S$ and $Q \leq T$, then $P \leq N_{(S,\phi, T)}$. We further set $N_\phi := N_{(N_G(P), \phi, N_H(Q))}$. Note that $N_\phi$ corresponds to $N_{\phi\inv}$ in the literature on fusion systems, see for example \cite{AKO11}.
	\end{definition}
	
	The following lemma is again routine. 
	
	\begin{lemma}\label{prop:proj&kernelsforspecialsubgroups}
		Let $G$ and $H$ be finite groups, let $\Delta(P, \phi, Q)$ be a twisted diagonal subgroup of $G\times H$, and let $C_G(P)\leq S\leq N_G(P)$ and $C_H(Q) \leq T \leq N_H(Q)$ be intermediate subgroups. 
		\begin{enumerate}
			\item One has $N_{G\times G}(\Delta(P)) = \Delta(N_G(P))\cdot (C_G(P) \times \{1\}) = \Delta(N_G(P)) \cdot (1\times C_G(P))$.
			\item For $X= N_{G\times H}(\Delta(P,\phi, Q))$, we have $k_1(X) = C_G(P)$, $k_2(X) = C_H(Q)$, $p_1(X) = N_\phi$, and $p_2(X) = N_{\phi\inv}$.
			\item For $X= N_{S\times T}(\Delta(P,\phi, Q))$, we have $k_1(X) = C_G(P)$, $k_2(X) = C_H(Q)$, $p_1(X) = N_{(S,\phi,T)}$, and $p_2(X) = N_{(T, \phi\inv, S)}$.
		\end{enumerate}
	\end{lemma}
	
	\section{$p$-permutation modules and the Brauer construction}
	
	Throughout this chapter, let $G$ be a finite group and let $(K,\calO, k)$ be a $p$-modular system with $\calO$ and $k$ having $|G|_{p'}$ roots of unity. Unless otherwise specified, $R$ will denote either $\calO$ or $k$. The properties listed in this section are well-known, we refer the reader to \cite{Br85} or \cite{Li181} for more details and proofs. 
	
	\begin{definition}
		An $RG$-module $M$ is a \textit{$p$-permutation module} if for every $p$-subgroup $P \leq G$, $\Res^G_P M$ is a permutation $RG$-module. An $(RG, RH)$-bimodule $M$ is a \textit{$p$-permutation bimodule} if it is a $p$-permutation module when considered as an $R[G\times H]$-module.
	\end{definition}

	\begin{theorem}
 
		\begin{enumerate}
			\item Let $M$ be a $RG$-module. The following are equivalent:
			\begin{enumerate}[label=(\roman*)]
				\item $M$ is a $p$-permutation module.
				\item Let $S \leq G$ be a Sylow $p$-subgroup of $G$. Then $\Res^G_S M$ is a permutation module.
				\item Each indecomposable direct summand of $M$ has trivial source.
			\end{enumerate}
			\item The functor induced by the surjection $\calO \to k$, $\overline{(-)}: {}_{\mathcal{O}G}\mathbf{mod} \to {}_{kG}\mathbf{mod}$ induces a vertex-preserving bijection from the set of isomorphism classes of $p$-permutation $\mathcal{O}G$-modules to the set of isomorphism classes of $p$-permutation $kG$-modules. 
			\item For any two $p$-permutation $\calO G$-modules $M, N$, the canonical map from $\Hom_{\calO G}(M, N)$ to $\Hom_{kG}(\overline{M}, \overline{N})$ is surjective. 
		\end{enumerate}
	\end{theorem}
	
	For this reason, indecomposable $p$-permutation modules are called \textit{trivial source modules}. We denote the full subcategory of $p$-permutation $RG$-modules by ${}_{RG}\textbf{triv}$, and the full subcategory of $p$-permutation $(RG,RH)$-bimodules by ${}_{RG}\textbf{triv}_{RH}$. Given an idempotent $e \in Z(RG)$, we may consider ${}_{RGe}\textbf{triv}$ a full subcategory of ${}_{RG}\textbf{triv}$. 
	
	\begin{definition}
		The Brauer construction is a functor $(-)(P): {}_{\calO G}\textbf{mod} \to {}_{k[N_G(P)/P]}\textbf{mod}$ defined on objects by: \[M(P) := M^P/\left(\sum_{Q < P} \text{tr}^P_Q(M^P) + J(\calO)M^P \right).\] The Brauer construction is also considered as a functor $(-)(P): {}_{kG}\textbf{mod} \to {}_{k[N_G(P)/P]}\textbf{mod}$ as well, via \[M(P) := M^P/\sum_{Q < P} \text{tr}^P_Q(M^P).\] This choice is convention is minor, as the following diagram of functors commutes for any $p$-subgroup $P \in s_p(G)$:
		
		\begin{figure}[H]
			\centering
			\begin{tikzcd}
			{}_{\calO G}\textbf{mod} \ar[d, "\overline{(-)}"] \ar[dr, "-(P)"] \\
			{}_{kG}\textbf{mod} \ar[r, "-(P)"]& {}_{kN_G(P)/P}\textbf{mod}
			\end{tikzcd}
		\end{figure}

		If $P$ is not a $p$-group, then $M(P)=0$ for every $RG$-module $M$. For this reason, we restrict our attention to taking Brauer quotients at $p$-groups. The Brauer construction restricts to a functor $(-)(P): {}_{RG}\triv \to {}_{k[N_G(P)/P]}\triv$. The Brauer construction is an additive functor, but in general it is not exact. Since it is additive, it can be extended to a functor of chain complexes, $\Ch({}_{RG}\textbf{mod}) \to \Ch({}_{k[N_G(P)/P]}\textbf{mod})$ and similarly for ${}_{RG}\triv$. 
	\end{definition}

    The following properties follow easily from the definitions. 
	
	\begin{prop}\label{brauercommuteswithdualsandconj}
		\begin{enumerate}
			\item Let $M$ be an $RG$-module and $P\in s_p(G)$. There is a natural isomorphism $M(P)^* \cong (M^*)(P)$  of $R[N_G(P)/P]$-modules. Similarly, if $M$ is an $(RG,RH)$-bimodule, then for any $X \in s_p(G\times H)$ there is a natural isomorphism $M(X)^* \cong (M^*)(X^\circ)$ of $R[N_{H\times G}(X^\circ)/X^\circ]$-modules. 
			\item Let $M$ be a $RG$-module, $P \in s_P(G)$, and $x \in G$. Then there is a natural isomorphism ${}^x(M(P)) \cong M({}^xP)$ of $k[{}^x(N_G(P)/P)]$-modules. 
		\end{enumerate}
	\end{prop}
	
	The Brauer construction behaves especially well with respect to $p$-permutation modules. The following properties are well-known, and are collected from \cite[\S 5.10-5.11]{Li181}.
	
	\begin{prop} \label{brauerppermproperties}
		\begin{enumerate}
			\item Let $M, N$ be $p$-permutation $RG$-modules and $P \in s_p(G)$. Then there is a natural isomorphism $(M\otimes_R N)(P) \cong M(P) \otimes_k N(P)$ of $k[N_G(P)/P]$-modules. 
			\item Let $M \cong RX$ be a permutation module and $P \in s_p(G)$. Then $M(P)$ has as a permutation $k$-basis the image of the $P$ fixed points $X^P$ under the quotient map $M^P \to M(P)$. In particular, $R[G/P](P) \cong k[N_G(P)/P]$ as $k[N_G(P)/P]$-modules.
			\item Let $M$ be a trivial source $RG$-module. Then $M(P) \neq 0$ if and only if $P$ is contained in a vertex of $M$.
			\item Let $Q \trianglelefteq P \in s_p(G)$ and let $M$ be a $p$-permutation $RG$-module. Then $M(P) \cong M(Q)(P)$, regarding $M(Q)$ as a $kN_G(Q)$-module. In particular, if $M(P) \neq 0$, then $M(R) \neq 0$ for any subgroup $R \leq P$.
			\item For each $P\in s_p(G)$, the Brauer construction $M \mapsto M(P)$ induces a bijection between the set of isomorphism classes of indecomposable $p$-permutation $RG$-modules with vertex $P$ and the set of isomorphism classes of indecomposable projective $k[N_G(P)/P]$-modules. Moreover, if $M$ is indecomposable with vertex $P$, then $M(P)$ is its Green correspondent. 
		\end{enumerate}
	\end{prop}

    The Brauer construction also is well-behaved with respect to vertices, and preserves twisted diagonal vertices.
    
    \begin{prop}{\cite[Lemma 3.7]{BoPe20}}
        \begin{enumerate}
            \item Let $M \in {}_{RG}\triv$ be indecomposable, and let $P \in s_p(G)$. Then each vertex of each indecomposable direct summand of $M(P)$ as $k[N_G(P)]$-module is contained in a vertex of $M$.
			\item Let $G$ and $H$ be finite groups and let $M \in {}_{RG}\triv_{RH}$ be indecomposable with twisted diagonal vertices. Let $X \in s_p(G\times H)$. Then each vertex of each indecomposable direct summand of the $k[N_{G\times H}(X)]$-module $M(X)$ is again twisted diagonal.  
        \end{enumerate}
    \end{prop}
	
	\begin{definition}
		The canonical map $M^P \to M(P)$ is denoted by $\Br^M_P$, or just $\Br_P$ if the context is clear, and called the \textit{Brauer map}. In the case where $M = RG$ and $G$ acts via conjugation, the map $\Br^{RG}_P$ translates to the projection map \[\br_P =\br^G_{P}: (RG)^P \to kC_G(P), \quad \sum_{g \in G}\alpha_g\cdot g \mapsto \sum_{g\in C_G(P)} \overline{\alpha_g} \cdot g .\] In fact, this is an algebra homomorphism, called the \textit{Brauer homomorphism}. Note that $\br_P(Z(RG))\subseteq Z(kC_G(P))^{N_G(P)}= Z(kN_G(P)) \cap kC_G(P)$. 
	\end{definition}

    We refer the reader to \cite[Part IV]{AKO11} for more details on Brauer maps. The final proposition allows us to compute the Brauer construction of a module that has been cut by a block idempotent.
	
	\begin{lemma}{\cite[Lemma 3.8]{BoPe20}}\label{lem:braur construction for iM}
		Let $M \in {}_{RG}\triv$, $P\in s_p(G)$, and $i \in (kG)^H$ an idempotent, fixed under the conjugation action of $H \leq G$ with $P \leq H$. Then for each $m \in M^P$, one has $\Br^M_P(im) = \br^G_P(i)\Br^M_P(m)$. In particular, one obtains a natural isomorphism $(iM)(P) \xrightarrow{\sim} \br^G_P(i)M(P)$ of $kN_H(P)$-modules.
		
	\end{lemma}

	\section{Block-theoretic preliminaries}
	
	Throughout, let $G$ and $H$ be finite groups and assume the $p$-modular system $(K,\calO, k)$ is large enough for $G \times H$ with $k$ algebraically closed for simplicity. We review some important preliminaries regarding Brauer pairs and block fusion systems. We assume the reader is familiar with the definition of a block, and that blocks of $\calO G$ and $kG$ are in one-to-one correspondence induced by the canonical surjection. If $B$ is a sum of blocks of $kG$, then we write $e_B$ for the identity element of $B$.
	
	\begin{remark}
		\begin{enumerate}
			\item Every block idempotent of $R[G\times H] = RG\otimes_R RH$ is of the form $e\otimes f$ for uniquely determined block idempotents $e \in Z(RG)$ and $f \in Z(RH)$. Note that we identify the $R$-algebras $R[G\times H] $ and $RG\otimes_R RH$ with $(g, h) \mapsto g\otimes h$.
			\item Recall that a \textit{Brauer pair} of $kG$ is a pair $(P,e)$, with $P \in s_p(G)$ and $e$ is a block idempotent of $kC_G(P)$. Note that the block idempotents of $k[C_G(P)]$ coincide with those of $k[PC_G(P)]$. Also, $G$ acts by conjugation on the set of Brauer pairs of $kG$ by ${}^g(P,e) = ({}^gP, {}^ge)$. We denote the $G$-stabilizer of $(P,e)$ by $N_G(P,e)$. We have $C_G(P)\leq N_G(P,e) \leq N_G(P)$. For Brauer pairs $(P,e)$ and $(Q,f)$ of $kG$, one writes $(Q,f)\trianglelefteq (P,e)$ if $Q \leq P \leq N_G(Q,f)$ and $\br_P(f)e = e$ (or equivalently, $\br_P(f)e \neq 0$). The transitive closure of the relation $\trianglelefteq $ on the set of Brauer pairs of $kG$ is denoted by $\leq$. It is a partial order and is preserved by $G$-conjugation.
			
			If $e$ is a central idempotent of $kG$, then $\br_P(e)$ is an $N_G(P)$-stable central idempotent of $kC_G(P)$ and also a central idempotent of $kN_G(P)$. If $B$ is a sum of blocks of $kG$, we say that a Brauer pair $(P,e)$ is a \textit{$B$-Brauer pair} if $\br_P(e_B)e = e$, or equivalently if $\br_P(e_B)e \neq 0$. Every Brauer pair is a $B$-Brauer pair for a unique block $B$ of $kG$, and in this case, $(\{1\}, e_B)\leq (P,e)$. The set of $B$-Brauer pairs is closed under conjugation and if $(Q,f) \leq (P,e)$ are Brauer pairs, $(Q,f)$ is a $B$-Brauer pair if and only if $(P,e)$ is. The set of $B$-Brauer pairs is denoted by $\calB\calP(B)$.
			
			For a sum $B$ of blocks of $\calO G$, we simply define $\calB\calP(B) = \calB\calP(\overline{B})$. Thus Brauer pairs are by default viewed as pairs $(P,e)$, where $e$ is an idempotent of a group algebra over $k$. It may sometimes be convenient to lift the idempotent to the corresponding idempotent over $\calO$, and we denote the set of resulting pairs by $\calB\calP_\calO(B)$.
			
			\item A \textit{defect group} of a block $B$ of $kG$ is a subgroup $D$ of $G$, minimal with respect to the property that $e_B \in \tr^G_D(kG^D)$. Equivalently, it is a subgroup $D$ for which $\Delta(D)$ is a vertex of $B$ viewed as a $k[G\times G]$-module. Equivalently, it is a $p$-subgroup $D$ of $G$, maximal with respect to the property that $\br_D(e_B) \neq 0$. The defect groups of $B$ form a $G$-conjugacy class of $p$-subgroups of $G$. If $P$ is a normal $p$-subgroup of $G$, then $P$ is contained in each defect group of each block $B$, and $e_B$ is contained in the $k$-span of the $p'$-elements of $C_G(P)$. Similarly one defines defect groups over $\calO$, and the same structural results hold. 
		\end{enumerate}
	\end{remark}
	
	We first state some well-known structural results, see \cite{BrP80} for details. 
	
	\begin{prop}
		\begin{enumerate}
			\item For each Brauer pair $(P,e)$ of $kG$ and each $Q\leq P$, there exists a unique Brauer pair $(Q,f)$ such that $(Q,f)\leq (P,e)$. In particular, if $(R, g) \leq (P,e)$ are Brauer pairs of $kG$, and $R\leq Q\leq P$, there exists a unique Brauer pair $(Q,f)$ of $kG$ satisfying $(R,g) \leq (Q,f)\leq (P,e)$.
			\item Let $(Q,f)$ and $(P,e)$ be Brauer pairs of $kG$. Then $(Q,f) \leq (P,e)$ and $Q \trianglelefteq P$ if and only if $(Q,f)\trianglelefteq (P,e)$.
			\item Let $B$ be a block of $kG$. The maximal elements in the poset $\calB\calP(B)$ form a single full conjugacy class. Further, a $B$-Brauer pair $(P,e)$ is a maximal $B$-Brauer pair if and only if $P$ is a defect group of $B$. 
		\end{enumerate}
	\end{prop}
	
	Note that $G\times G$ acts on $kG$ by the $k[G\times G]$-module structure, and that $G$ acts on $kG$ by $G$-conjugation. These actions are linked by the diagonal embedding $\Delta: g\mapsto (g,g)$, so that $kG^{\Delta (H)} = kG^{H}$. The following are easy verifications. 
	
	\begin{prop}
		Let $B$ be a sum of blocks of $kG$ and let $(Q,e)$ be a $B$-Brauer pair. Set $I := N_G(Q,e)$. 
		\begin{enumerate}
			\item We have $B(\Delta(Q))\cong k[C_G(Q)]\br_Q(e_B)$ as $k[N_{G\times G} (\Delta(Q))]$-modules.
			\item We have $eB(\Delta(Q))e \cong k[C_G(Q)]e$ as $k[N_{I\times I}(\Delta(Q))]$-modules. 
		\end{enumerate}
	\end{prop}
    \begin{proof}
        This is \cite[Proposition 4.3]{BoPe20}, with the unused assumption that $B$ is a block algebra removed. 
    \end{proof}
	
	\begin{remark}
		We next recall the definition of a fusion system of a block. See \cite{AKO11} for more details, proofs, and the definition of an abstract fusion system and a saturated fusion system. Let $B$ be a block of $kG$ and let $(P,e)$ be a maximal $B$-Brauer pair. For $Q \leq P$, denote by $e_Q$ the unique block idempotent of $kC_G(Q)$ for which $(Q,e_Q)\leq (P,e)$.  
		
		\textit{The fusion system of $B$} associated to $(P,e)$ is the category $\calF$ whose objects are subgroups of $P$ and whose morphism set $\Hom_\calF(Q,R)$ is the set of group homomorphisms arising as conjugation maps $c_g: Q\to R$ where $g \in G$ satisfies ${}^g(Q,e_Q) \leq (R,e_R)$. It is well known that $\calF$ is a saturated fusion system on $P$.
		
		For an abstract fusion system $\calF$ on a $p$-group $P$, a subgroup $Q\leq P$ is \textit{fully $\calF$-centralized} (resp. \textit{fully $\calF$-normalized}) if $|C_P(Q)| \geq |C_P(Q')|$ for all (resp. $|N_P(Q)| \geq |N_P(Q')|$) for all subgroups $Q'$ of $P$ which are $\calF$-isomorphic to $Q$. Moreover, $Q$ is \textit{$\calF$-centric} if $C_P(Q') = Z(Q')$ for all $Q'$ that are $\calF$-isomorphic to $Q$.
	\end{remark}
	
	The following theorem of Alperin and Brou\'e in \cite{AB79} characterizes $\calF$-centric, fully normalized, and fully centralized subgroups for block fusion systems. We give the formulation presented in \cite{L06}.
	
	\begin{prop}{\cite[Theorems 2.4, 2.5]{L06}}
		Let $B$ be a block of $kG$, let $(P,e)$ be a maximal $B$-Brauer pair, and let $\calF$ be the fusion system associated to $B$ and $(P,e)$. For every subgroup $Q\leq P$, denote by $e_Q$ the unique block idempotent of $k[C_G(Q)]$ for which $(Q,e_Q)\leq (P,e)$. 
		\begin{enumerate}
			\item A subgroup $Q$ of $P$ is fully $\calF$-centralized if and only if $C_P(Q)$ is a defect group of the block algebra $k[C_G(Q)]e_Q$. In this case, $(C_P(Q), e_{QC_P(Q)})$ is a maximal Brauer pair of $k[C_G(Q)]e_Q$. In particular, $Q$ is $\calF$-centric if and only if $Z(Q)$ is a defect group of $k[C_G(Q)]e_Q$.
			\item A subgroup $Q$ of $P$ is fully $\calF$-normalized if and only if $N_P(Q)$ is a defect group of the block algebra $k[N_G(Q,e_Q)]e_Q$. In this case, $(N_P(Q), e_{N_P(Q)})$ is a maximal $k[N_G(Q,e_Q)]e_Q$-Brauer pair. 
		\end{enumerate}
	\end{prop}
	
	\section{Trivial source rings, $p$-permutation equivalences, and splendid Rickard equivalences}

    Throughout this section we assume that $G$ and $H$ are finite groups and that $(K, \calO, k)$ is a $p$-modular system large enough for $G$ and $H$. We denote by $R$ either $\calO$ or $k$. Let $e$ be a central idempotent of $RG$ and let $f$ be a central idempotent of $RH$.
 
	\begin{definition}
		\begin{enumerate}
			\item We denote by $T(RGe)$ the \textit{trivial source group} for $RGe$, the Grothendieck group of the category ${}_{RGe}\textbf{triv}$ with respect to direct sums. The group $T(RGe)$ is free as an abelian group with standard basis given by the elements $[M]$, where $M$ runs through a set of representatives of isomorphism classes of trivial source $RGe$-modules. If $e = 1$, then $T(RG)$ forms a commutative ring under multiplication induced by $\otimes_R$, denoted $\cdot_R$; we call $T(RG)$ the \textit{trivial source ring}. We always view $T(RGe)$ as a subgroup of $T(RG)$ in the natural way. Moreover, multiplying a $p$-permutation $RG$-module by $e$ induces a projection $T(RG)\to T(RGe)$. In this way, if $b_1, \dots, b_k$ are the blocks of $RG$, then we have a direct sum decomposition of abelian groups $T(RG) = T(RGb_1) \oplus \cdots \oplus T(RGb_k)$.
			
			Taking duals induces a self-inverse group isomorphism $T(RG)\to T(RG)$. We write $O(T(RG))$ for the \textit{orthogonal unit group} of the trivial source ring, that is, the subgroup of $T(RG)^\times$ consisting of all virtual $p$-permutation modules $\gamma$ whose inverse is its dual, i.e. $\gamma \cdot_R \gamma^* = [R] = \gamma^* \cdot_R \gamma$. This notion does not extend to idempotents because if $\gamma \in O(T(RGe))$, then $\gamma^* \in O(T(RGe^*))$, where $(-)^*$ is the map induced by the involution $g\mapsto g\inv$. 
						
			\item Similarly, we define $T(RGe,RHf): = T(R[G\times H])(e\otimes f^*)$ where we identify $R[G\times H]$ with $RG\otimes_R RH$ in the natural way. In the case where $e=f$, $T(RGe,RGe)$ has a ring structure induced by the tensor product $\otimes_{RGe}$, written $\cdot_G$. We denote by $T^\Delta(RGe, RHf)$ the subgroup of $T(RGe, RHf)$ the subgroup generated by all trivial source $(RGe, RHf)$-bimodules which have twisted diagonal vertices.  
			Taking bimodule duals induces a group isomorphism $T(RGe, RHf) \to T(RHf, RGe)$, which has inverse again induced by taking duals. Denote by $O(T(RGe,RGe))$ the subgroup of $T(RGe,RGe)^\times$ consisting of all virtual $p$-permutation bimodules $\gamma$ whose inverse is its dual, i.e. $\gamma \cdot_G \gamma^* = [RGe] = \gamma^* \cdot_G \gamma$. 
			
			Finally, observe that any $M \in T(RGe,RHf)$ induces a group homomorphism $M\cdot_{H} -: T(RHf)\to T(RGe)$. If $M \in T(RGe, RHf)^\times$, it is an isomorphism. 
			
			\item The following definition is due to Boltje and Xu in \cite{BoXu07}. A \textit{$p$-permutation equivalence between $RGe$ and $RHf$ is an element $\gamma \in T^\Delta(RGe, RHf)$} satisfying \[\gamma \cdot_H \gamma^* = [RGe] \in T^\Delta(RGe, RGe) \text{ and } \gamma^*\cdot_G\gamma = [RHf] \in T^\Delta(RHf, RHf).\] Denote the set of $p$-permutation equivalences between $RGe$ and $RHf$ by $T^\Delta_o(RGe, RHf)$. In the case $e=f$, $T^\Delta_o(RGe, RGe) = O(T(RGe,RGe))\cap T^\Delta(RGe,RGe)$. We denote this by $O(T^\Delta(RGe,RGe))$ as well. It follows that any $p$-permutation equivalence between $RGe$ and $RHf$ induces a group isomorphism $T(RGe) \cong T(RHf)$. 
		\end{enumerate}
	\end{definition}

    If $\calA$ is an additive category, then we denote by $\Ch^b(\calA)$ the category of bounded chain complexes in $\calA$ and by $K^b(\calA)$ the bounded homotopy category of chain complexes in $\calA$. 
     
	\begin{definition}\label{def:splendid rickard}
		Let $A$ be a sum of blocks of $RG$ and $B$ a sum of blocks of $RH$. A \textit{splendid Rickard equivalence} for $A$ and $B$ is a bounded chain complex $C$ of $p$-permutation $(A, B)$-bimodules which satisfy the following conditions:
		\begin{enumerate}
             \item \label{def:splendid rickard (a)} For every $n\in \Z$, each indecomposable direct summand of $C_n$ has a twisted diagonal subgroup of $G\times H$ as a vertex. 
			\item $C \otimes_{B} C^* \simeq A$ in $K^b({}_A\mathbf{mod}_A)$ and $C^* \otimes_{A} C \simeq B$ in $K^b({}_B\mathbf{mod}_B)$, where $A$ and $B$ are considered as chain complexes concentrated in degree $0$.
		\end{enumerate}
		A splendid Rickard equivalence $C$ of $(A, B)$-bimodules induces an equivalence of categories $C \otimes_{B} -: K^b({}_{B}\textbf{triv}) \to K^b({}_{A}\textbf{triv})$. 
	\end{definition}
 
    % The next lemma is a general fact which will be used primarily in Section 7. We state it here because it provides a clean proof of Proposition~\ref{prop:src indecomposable}.
    
    \begin{lemma}\label{lem:contractible tensor product of complexes}Let $A$ and $B$ be symmetric $R$-algebras. Suppose that $C$ is a bounded chain complex of $(A,B)$-bimodules which are projective as right $B$-modules. Then $C\otimes_B C^* \simeq 0$ in $K^b({}_A \mathbf{mod}_A)$ if and only if $C\simeq 0$ in $K^b({}_A \mathbf{mod}_B)$.
    
    \begin{proof}
        Since $B$ is symmetric and the terms of $C$ are projective as right $B$-modules, we have an isomorphism $C\otimes_B C^* \cong \mathrm{End}_{B^{op}}(C)$ of chain complexes of $(A,A)$-bimodules. First, suppose $\mathrm{End}_{B^{op}}(C)$ is contractible. The complex $\mathrm{End}_{A\otimes_k B^{op}}(C)$ is obtained from the complex $\mathrm{End}_{B^{op}}(C)$ by taking $A$-fixed points. Since taking $A$-fixed points is an additive functor, $\mathrm{End}_{A\otimes_R B^{op}}(C)$ is also contractible. Taking homology we obtain
        \[
        0 = H_0(\mathrm{End}_{A\otimes_R B^{op}}(C)) \cong\End_{K({}_{A}\mathbf{mod}_{B})}(C).
        \]
        We conclude that $C\simeq 0$ in $K({}_{A}\mathbf{mod}_{B})$. The converse is immediate. 
        \end{proof}
        \end{lemma}

        If $A$ and $B$ are blocks of $RG$ and $RH$, then a splendid Rickard equivalence for $A$ and $B$ is an indecomposable object in the homotopy category of chain complexes of $(A,B)$-bimodules. 
        
    \begin{prop}\label{prop:src indecomposable}
        Suppose that $A$ is a block of $RG$ and $B$ is a block of $RH$. Let $C$ be a splendid Rickard equivalence for $A$ and $B$. Then $C$ is indecomposable in $K^b({}_A\mathbf{mod}_B)$. In particular, $C$ has, up to isomorphism, a unique indecomposable noncontractible direct summand which is a splendid Rickard equivalence for $A$ and $B$.
        
    \end{prop}
    \begin{proof}By Lemma~\ref{lem:contractible tensor product of complexes}, $C$ is not contractible. Using the Krull-Schmidt theorem in $\Ch^b({}_A\mathrm{mod}_B)$, write $C = \bigoplus_{i=1}^n C_i$ where each $C_i$ is indecomposable. Then
    \[
    A\simeq C\otimes_B C^* = \bigoplus_{1\leq i,j\leq n}C_i \otimes C_j^*.
    \]
    Since $A$ is indecomposable as an $(A,A)$-bimodule, the Krull-Schmidt theorem in $K^b({}_A\mathrm{mod}_A)$ implies that there is a unique pair $(i,j)$ such that $C_i\otimes_B C_j^* \simeq A$ and $C_k\otimes_B C_l^*\simeq 0$ for all $(k,l) \neq (i,j)$. Thus, if $i \neq j$, then $C_k\otimes_B C_k^* \simeq 0$ for all $k=1,\ldots, n$ which forces $C_k\simeq 0$ by Lemma~\ref{lem:contractible tensor product of complexes}. This implies that $C\simeq 0$, which is a contradiction. We conclude that $C_i\otimes_B C_i^* \simeq A$ and $C_k\simeq 0$ for all $k\neq i$. Thus, $C\simeq C_i$ is indecomposable in $K^b({}_{A}\mathbf{mod}_B)$. Using the other homotopy equivalence, we also obtain that $B\simeq C^*\otimes_A C = C_i^*\otimes_A C_i$ so that $C_i$ is a splendid Rickard equivalence for $A$ and $B$.
        \end{proof}
 
    Historically, the notion of a ``splendid'' (\textbf{SPL}it \textbf{END}omorphism two-sided tilting complex \textbf{I}nduced from \textbf{D}iagonal subgroups) Rickard equivalence is due to Rickard, and was additionally communicated by Brou\'e. In Rickard's paper \cite{Ri96} such complexes are referred to as ``splendid equivalences,'' and the nomenclature ``Rickard complex'' was commonly used prior to refer to two-sided tilting complexes (see \cite{Br93} for example, in which Rickard's original definition also appears) due to his landmark paper \cite{R91}. Rickard's original definition assumed that $G$ and $H$ shared a common Sylow $p$-subgroup, and that the vertices of the modules in a splendid Rickard equivalence are contained inside the diagonal subgroup generated by the common Sylow. We remove that assumption here, instead asserting only that the modules are twisted diagonal. 
    
    Rickard proved that under his original definition, in order for a complex between two \emph{blocks} to induce a splendid equivalence, it is sufficient for only one of the homotopy equivalences to be satisfied. This argument extends to the more general definition given here. Note that the following result assumes $A$ and $B$ are single blocks, but the statement and argument can be modified to work for sums of blocks as well. 

    \begin{prop}\label{prop:src left implies right}
        Assume that $A$ and $B$ are blocks of $RG$ and $RH$, respectively. Let $C$ be a bounded chain complex of $p$-permutation $(A,B)$-bimodules which satisfies condition \ref{def:splendid rickard}\ref{def:splendid rickard (a)}. If $C\otimes_B C^* \simeq A$ in $K^b({}_A\mathbf{mod}_A)$, then $C$ is a splendid Rickard equivalence. 
    \begin{proof}
        An indecomposable $(A,B)$-bimodule $M$ with a twisted diagonal vertex $X$ is projective as a left $A$-module and a right $B$-module. Indeed, by the Mackey Formula and the Krull-Schmidt theorem, every indecomposable direct summand of $\mathrm{Res}_{G\times 1}^{G\times H}(M)$ is projective relative to $(G\times 1)\cap {}^{(g,h)}X$ for some $(g,h) \in G\times H$. Since ${}^{(g,h)}X$ is also twisted diagonal, this intersection is trivial. This shows that $M$ is projective as a left $A$-module. A similar argument shows that $M$ is projective as a right $B$-module. Hence, the result follows from \cite[4.14.16]{Li181}.
    \end{proof}
        
    \end{prop}

	Boltje and Xu proved that $p$-permutation equivalences lie between isotypies and splendid equivalences in the following way. Given a bounded chain complex $C$ of $p$-permutation modules, the \textit{Lefschetz invariant} of $C$ is 
    \[
    \Lambda(C) := \sum_{n\in \Z} (-1)^n[C_n] \in T(RG).
    \]
	
	\begin{theorem}{\cite[Theorem 1.5]{BoXu07}}
		Let $C$ be a splendid Rickard equivalence between $A$ and $B$. Then $\gamma := \Lambda(C)$ is a $p$-permutation equivalence between $RGe$ and $RHf$.
	\end{theorem}
	
	\begin{remark}
        Under Rickard's original definition, if $C$ is a splendid Rickard equivalence for $A$ and $B$, for any $Q \leq P$ and blocks $e_Q$ and $f_Q$ corresponding to the blocks of $kC_G(Q)$ and $kC_H(Q)$ determined by $A$ and $B$, the chain complex $e_Q C(\Delta Q) f_Q$ is a splendid Rickard equivalence for $kC_G(Q)e_Q$ and $kC_H(Q)f_Q$. Such equivalences are regarded as ``local'' equivalences, equivalences between blocks of local centralizer subgroups. 
	\end{remark}

	\section{Extended tensor products and tensor products of $p$-permutation bimodules}
	
	In this section, we explain some of the tensor products that will arise in the sequel. Throughout, $G,H,K$ are finite groups, and $R$ denotes a commutative ring. 
	
	\begin{definition}
		
		Let $X \leq G\times H$ and $Y \leq H \times K$. Furthermore, let $M \in {}_{RX}\catmod$ and $N \in {}_{RY}\catmod$. Since $k_1(X) \times k_2(X) \leq X$, $M$ can be viewed as $(R[k_1(X)], R[k_2(X)])$-bimodule. Similarly, $N$ can be viewed as $(R[k_1(Y)], R[k_2(Y)])$-bimodule, so $M\otimes_{R[k_1(X) \cap k_2(Y)]}N$ is a $(R[k_1(X)], R[k_2(Y)])$-bimodule. Note that $k_1(X) \times k_2(Y) \leq X\ast Y$, and that this bimodule structure can be extended to a $R[X\ast Y]$-module structure such that for $(g,k) \in X\ast Y,$ $m \in M$, $n \in N$, we have \[(g,k)\cdot (m\otimes n) = (g,h)m \otimes (h,k)n, \] where $h \in H$ is chosen such that $(g,h) \in X$ and $(h, k)\in Y$. Note this is independent of choice of $h$. This construction was first used in \cite{Bo10b}. 
        We denote this \textit{extended tensor product} by $M \underset{RH}{\overset{X,Y}{\otimes}} N \in {}_{R[X\ast Y]}\catmod$, and obtain a bifunctor $- \underset{RH}{\overset{X,Y}{\otimes}} -: {}_{RX}\catmod \times {}_{RY}\catmod \to {}_{R[X\ast Y]}\catmod$. 
		
		Some quick calculations determine that this construction is associative in the following way: if $L$ is a finite group, $Z \leq K \times L$, and $P \in {}_{RZ}\catmod$, then \[(M \underset{RH}{\overset{X,Y}{\otimes}} N) \underset{RK}{\overset{X\ast Y,Z}{\otimes}}P \xrightarrow{\sim} M \underset{RH}{\overset{X,Y\ast Z}{\otimes}} (N \underset{RK}{\overset{Y,Z}{\otimes}}P ), \quad (m\otimes n)\otimes p \mapsto m \otimes(n\otimes p).\] 
  defines a natural isomorphism of $R[X\ast Y \ast Z]$-modules which is natural in $M,N$ and $K$.
		
		Moreover, for $g\in G$, $h \in H$, $k\in K$, we have a natural isomorphism \[ {}^{(g,k)} (M \underset{RH}{\overset{X,Y}{\otimes}} N) \cong {}^{(g,h)}M \underset{RH}{\overset{{}^{(g,h)}X,{}^{(h,k)}Y}{\otimes}} {}^{(h,k)}N \] of $R[{}^{(g,k)} (X\ast Y)]$-modules.

	\end{definition}

	The following theorem of Serge Bouc describes an interaction between additive induction, standard tensor products, and extended tensor products. 
	
	\begin{theorem}{\cite{Bo10b}}\label{bouciso}
		Let $X \leq G\times H$, $Y\leq H\times K$, $M \in {}_{RX}\mathbf{mod}$, $N \in {}_{RY}\mathbf{mod}$. There is a natural isomorphism
  \[\Ind^{G\times H}_X(M) \otimes_{RH} \Ind^{H\times K}_Y(N) \cong \bigoplus_{t \in [p_2(X) \backslash H / p_1(Y)]} \Ind^{G\times K}_{X \ast {}^{(t,1)}Y} (M \mathop{\otimes}\limits_{RH}^{X, {}^{(t,1)}Y} {}^{(t,1)}N)\] of $(RG, RK)$-bimodules.
	\end{theorem}
	
	We next turn our attention to extended tensor products of $p$-permutation modules in conjunction with the Brauer construction. In this case, the tensor products are particularly well-behaved.
	
	\begin{lemma}{\cite[Lemma 7.2]{BoPe20}}
		Suppose $(K, \calO, k)$ is a $p$-modular system. Let $X \leq G\times H$ and $Y \leq H \times K$. 
		\begin{enumerate}
			\item If $M \in {}_{\calO X}\triv$ and $N \in {}_{\calO Y}\triv$, then $M \overset{X,Y}{\underset{\calO H}{\otimes}} N \in {}_{\calO [X\ast Y]}\triv$. 
			\item If $M \in {}_{\calO X}\triv$ and $N \in {}_{\calO Y}\triv$ are indecomposable with twisted diagonal vertices, then each indecomposable direct summand of the $\calO [X\ast Y]$-module $M \overset{X,Y}{\underset{\calO H}{\otimes}} N $ has twisted diagonal vertices. 
		\end{enumerate}
	\end{lemma}
	
	\begin{remark}
		This construction originally appeared in \cite{BoDa12}. For a fixed twisted diagonal $p$-subgroup $\Delta(P, \sigma, R)\leq G\times K$, denote by $\Gamma_H(P,\sigma, R)$ the set of triples $(\phi, Q,\psi)$ where $Q \leq H$ and $\psi: R\xrightarrow{\sim} Q$ and $\phi: Q\xrightarrow{\sim}P$ are isomorphisms with $\phi\circ \psi = \sigma$. The group $H$ acts on $\Gamma_H(P,\sigma, R)$ by \[{}^h(\phi, Q,\psi) := (\phi c_h\inv , {}^h Q, c_h \psi).\] Note that $\stab_H(\phi, Q,\psi) = C_H(Q)$. For $M \in {}_{\calO G}\catmod_{\calO H}$, $N \in {}_{\calO H}\catmod_{\calO K}$, and $(\phi, Q,\psi) \in \Gamma_H(P,\sigma, R)$, one has a $(k[C_G(P)], k[C_K(R)])$-bimodule homomorphism 
		\begin{align*}
		\Phi_{(\phi, Q,\psi)}: M(\Delta(P,\phi, Q)) \otimes_{k[C_H(Q)]} N(\Delta(Q,\psi, R)) &\to (M\otimes_{\calO H} N)(\Delta(P,\sigma, R))\\
		\Br_{\Delta(P,\phi, Q)}(m)\otimes \Br_{\Delta(Q,\psi, R)}(n) &\mapsto \Br_{\Delta(P,\sigma, R)}(m\otimes n),
		\end{align*}
		(see \cite[3.1]{BoDa12}) which is natural in $M$ and $N$. The following theorem characterizes the tensor product on the right-hand side for a particular class of $p$-permutation modules.
	\end{remark}
	
	\begin{theorem}{\cite[Theorem 3.3]{BoDa12}}
		Let $\Delta(P, \sigma, R) \in s_p(G\times K)$ be a twisted diagonal $p$-subgroup, let $\Gamma = \Gamma_H(P,\sigma ,R)$ be as above, and let $\widetilde\Gamma \subseteq \Gamma$ be a set of representatives of the $H$-orbits of $\Gamma$. Furthermore, let $M \in {}_{\calO G}\triv_{\calO H}$ and $N \in {}_{\calO H}\triv_{\calO K}$ be $p$-permutation bimodules whose indecomposable direct summands have twisted diagonal vertices. Then the direct sum of the homomorphisms $\Phi_{(\phi, Q,\psi)}$ for $(\phi,Q,\psi)\in \widetilde{\Gamma}$ yields an isomorphism \[\Phi: \bigoplus_{(\phi,Q,\psi)\in \widetilde{\Gamma}} M(\Delta(P,\phi, Q)) \otimes_{k[C_H(Q)]} N(\Delta(Q,\psi, R)) \to (M\otimes_{\calO H} N)(\Delta(P,\sigma, R)) \] of $(k[C_G(P)],k[C_K(R)])$-bimodules which is natural in $M$ and $N$.
	\end{theorem}
	
	Note that not only $C_G(P)\times C_K(R)$, but also the bigger group $N_{G\times K}(\Delta(P,\sigma ,R))$ acts on the right-hand side of the isomorphism above. Therefore, we may consider the extended tensor product structure on the right-hand side and attempt to transport it over to the left. This was done by Boltje and Perepelitsky in the following way. First, note the domain of the homomorphism $\Phi_{(\phi, Q, \psi)}$ carries a $k[N_{G\times H}(\Delta(P,\phi, Q))\ast N_{H\times K}(\Delta(Q,\psi, R))]$-module structure via the extended tensor product construction. This extends the $k[C_G(P) \times C_K(R)]$-module structure from before, since $k_1(N_{G\times H}(\Delta(P,\phi, Q))) = C_G(P)$ and $k_2(N_{G\times H}(\Delta(P,\phi, Q))) = C_H(Q) = k_1(N_{H\times K}(\Delta(Q,\psi, R)))$, and $k_2(N_{H\times K}(\Delta(Q,\psi, R)))= C_K(R)$. It follows that $\Phi_{(\phi, Q, \psi)}$ actually defines a homomorphism \[\Phi_{(\phi, Q,\psi)}: M(\Delta(P,\phi, Q)) \overset{N_{G\times H}(\Delta(P,\phi, Q)), N_{H\times K}(\Delta(Q,\psi, R))}{\underset{kH}{\otimes}} N(\Delta(Q,\psi, R)) \to (M\otimes_{\calO H} N )(\Delta(P,\sigma, R)) \] of $k[N_{G\times H}(\Delta(P,\phi, Q))\ast N_{H\times K}(\Delta(Q,\psi, R))]$-bimodules which is natural in $M$ and $N$. 
	
	\begin{corollary}{\cite[Corollary 7.4]{BoPe20}}
		Let $M \in {}_{\calO G}\triv_{\calO H}, N\in {}_{\calO H}\triv_{\calO K}$, $\Delta (P,\sigma, R)\in s_p(G\times K)$, and $\widetilde\Gamma \subseteq \Gamma = \Gamma_H(P,\sigma, R)$ be as before. 
		\begin{enumerate}
			\item The group $N_{G\times K}(\Delta(P,\sigma,R))\times H$ acts on $\Gamma$ via \[{}^{((g,k),h)}(\phi, Q,\psi) := (c_g\phi c_h\inv, {}^hQ, c_h\psi c_k\inv).\] For the induced action of $N_{G\times K}(\Delta(P,\sigma,R))$ on the $H$-orbits $[\phi, Q,\psi]_H$ of $\Gamma$, one has \[\stab_{N_{G\times K}(\Delta(P,\sigma,R))}([\phi, Q,\psi]_H) = N_{G\times H}(\Delta(P,\phi, Q))\ast N_{H\times K}(\Delta(Q,\psi, R))\] for each $(\phi,Q,\psi)\in \Gamma$. 
   
			\item Let $\widehat\Gamma \subseteq \widetilde\Gamma$ be a set of representatives of the $N_{G\times K}(\Delta(P,\sigma, R))\times H$-orbits of $\Gamma$. The homomorphisms $\Phi_{(\phi, Q, \psi)}$ for each $(\phi, Q,\psi) \in \widehat{\Gamma}$ induce an isomorphism 
			\[\Phi: \bigoplus_{\gamma = (\phi,Q,\psi) \in \widehat{\Gamma}} \Ind_{X(\gamma)\ast Y(\gamma)}^{N_{G\times K}(\Delta(P,\sigma, R))} \left( M(\Delta(P,\phi, Q)) \overset{X(\gamma),Y(\gamma)}{\underset{kH}{\otimes}} N(\Delta(Q,\psi, R)) \right) \xrightarrow{\sim} (M\otimes_{\calO H} N)(\Delta(P,\sigma, R))   \]
			of $k[N_{G\times K}(\Delta(P,\sigma, R))]$-modules which is natural in $M$ and $N$, with $X(\gamma) := N_{G\times H}(\Delta(P,\phi, Q))$ and $Y(\gamma) :=  N_{H\times K}(\Delta(Q,\psi, R))$ for any $\gamma = (\phi, Q,\psi) \in \Gamma$.
		\end{enumerate}
	\end{corollary}
	
	This theorem has a block-wise version as well. First, observe that for any twisted diagonal subgroup $\Delta(P,\phi, Q)$ of $G\times H$, one has $C_{G\times H}(\Delta(P,\phi, Q)) = C_G(P)\times C_H(Q)$, and that for Brauer pairs $(P,e), (Q,f)$ for $kG$ and $kH$ respectively, the pair $(\Delta(P, \phi, Q), e\otimes f^*)$ is a Brauer pair for $k[G\times H] = kG \otimes_k kH$.
	
	\begin{theorem}{\cite[Theorem 7.5]{BoPe20} }\label{7.5}
		Let $(P,e)\in \mathcal{BP}(kG)$ and $(R,d)\in \mathcal{BP}(kK)$. Suppose $\sigma: R\xrightarrow{\sim} P$ is an isomorphism and let $C_G(P)\leq S\leq N_G(P,e)$ and $C_K(R)\leq T\leq N_K(R,d)$ be intermediate subgroups. Further, let $\Omega:= \Omega_H((P,e), \sigma, (R,d))$ denote the set of triples $(\phi, (Q,f), \psi)$, where $(Q,f)\in\mathcal{BP}(kH)$ and $\psi: R\xrightarrow{\sim} Q$ and $\phi: Q\xrightarrow{\sim} P$ are isomorphisms such that $\sigma = \phi \circ \psi$. Finally, let $M \in {}_{\calO G}\triv_{\calO H}$ and $N \in {}_{\calO H}\triv_{\calO K}$ be $p$-permutation modules all of whose indecomposable direct summands have twisted diagonal vertices.
		\begin{enumerate}
			\item The group $N_{G\times K}(\Delta(P,\sigma,R))\times H$ acts on $\Omega$ via \[{}^{((g,k),h)}(\phi, (Q,f),\psi) = (c_g\phi c_h\inv, {}^h(Q,f), c_h\psi c_k\inv)\] and $\stab_H(\phi, (Q,f),\psi) = C_H(Q)$.
			
			\item Let $\widetilde{\Omega} \subseteq \Omega$ be a set of representatives of the $H$-orbits of $\Omega$. The restrictions $\Phi_{(\phi, (Q,f), \psi)}$ of $\Phi_{(\phi, Q, \psi)}$ to $eM(\Delta(P,\phi, Q))f\otimes_{k[C_H(Q)]} fN(\Delta(Q,\psi, R))d$ define an isomorphism
			\begin{align*}
			\Phi: \bigoplus_{(\phi, (Q,f),\psi)\in \widetilde{\Omega}} eM(\Delta(P,\phi, Q))f\otimes_{k[C_H(Q)]} fN(\Delta(Q,\psi, R))d &\cong e(M\otimes_{\calO H} N)(\Delta(P,\sigma, R))d\\
			e\mathrm{Br}_{\Delta(P,\phi,Q)}(m)f\otimes f\mathrm{Br}_{\Delta(Q,\psi,R)}(n)d &\mapsto e\mathrm{Br}_{\Delta(P,\sigma, R)}(m\otimes n)d
			\end{align*}
			of $(k[C_G(P)]e, k[C_K(R)]d)$-bimodules which is natural in $M$ and $N$.
			
			\item Let $\widehat{\Omega} \subseteq \widetilde{\Omega}\subseteq \Omega$ be a set of representatives of the $N_{S\times T}(\Delta(P,\sigma, R))\times H$ orbits of $\Omega.$ The homomorphisms $\Phi_{(\phi, (Q,f), \psi)}: eM(\Delta(P,\phi, Q))f\otimes_{k[C_H(Q)]} fN(\Delta(Q,\psi, R))d \to e(M\otimes_{\calO H} N)(\Delta(P,\sigma, R))d$ for $(\phi, (Q,f), \psi) \in \widehat\Omega$ induce an isomorphism
			
			\[\Phi: \bigoplus_{\omega\in \hat{\Omega}} \Ind^{N_{S\times T}(\Delta(P,\sigma, R))}_{X(\omega)\ast Y(\omega)} (M(\omega)\mathop{\otimes}_{kH}^{X(\omega),Y(\omega)} N(\omega)) \cong e(M\otimes_{\calO H} N)(\Delta(P,\sigma, R))d\] 
			
			of $N_{S\times T}(\Delta(P,\sigma, R))$-modules which is natural in $M$ and $N$. Here, for $\omega = (\phi, (Q,f), \psi) \in \widehat{\Omega}$, we set $X(\omega) := N_{S\times H}(\Delta(P,\phi, Q), e\otimes f^*), Y(\omega) := N_{H\times T} (\Delta(Q,\psi, R), f\otimes d^*), M(\omega) := eM(\Delta (P, \phi, Q))f,$ and $N(\omega) := fN(\Delta(Q,\psi, R))d$.
			
			\item Assume that $G = K$, $(P,e) = (R,d)$, $S = T$, $\sigma = \text{id}_P$. Let $\Lambda := \Lambda_H(P)$ be the set of pairs $(\phi, (Q,f))$ where $(Q,f)$ is a Brauer pair of $kH$ and $\phi: Q\to P$ is an isomorphism. The group $S \times H$ acts on $\Lambda$ via \[{}^{(g,h)}(\phi, (Q,f)) = (c_g\phi c_h\inv, {}^h(Q,f)),\] for $(g,h) \in S\times H$ and $(\phi, (Q,f)) \in \Lambda. $ Let $\hat{\Lambda} \subseteq \widetilde{\Lambda} \subseteq \Lambda$ be such that $\widetilde{\Lambda}$ (resp. $\hat{\Lambda}$) is a set of representatives of the $H$-orbits (resp. $S\times H$-orbits) of $\Lambda$. Then one has a natural isomorphism 
			
			\[e(M\otimes_{\calO H} N)(\Delta(P))e \cong \bigoplus_{(\phi, (Q,f))\in \widetilde{\Lambda}} eM(\Delta(P,\phi, Q))f \otimes_{k[C_H(Q)]} fN(\Delta(Q, \phi\inv, P))e\]
			
			of $(k[C_G(P)]e, k[C_G(P)]e)$-bimodules and a natural isomorphism 
			
			\[e(M\otimes_{\calO H} N)(\Delta(P))e \cong \bigoplus_{\lambda = (\phi, (Q,f))\in \widehat{\Lambda}} \Ind^{N_{S\times S}(\Delta(P))}_{\Delta(I(\lambda))(C_G(P)\times 1)} (M(\lambda) \mathop{\otimes}_{kH}^{X(\lambda),Y(\lambda)} N(\lambda))\] 
			
			of $k[N_{S\times S} (\Delta (P))](e\otimes e^*)$-modules. Here, for $\lambda = (\phi, (Q,f)) \in \Lambda$, we set $X(\lambda) := N_{S\times H}(\Delta(P,\phi, Q), e\otimes f^*)$, $Y(\lambda) := N_{H\times S}(\Delta(Q, \phi\inv, P), f\otimes e^*)$, $M(\lambda) := eM(\Delta(P,\phi, Q))f,$ $N(\lambda) := fN(\Delta(Q,\phi\inv, P))e$, and $I(\lambda) := N_{(S, \phi, N_H(Q,f))}$.
		\end{enumerate}
	\end{theorem}

    Note that the previous theorems also hold over base ring $k$ replacing $\calO$. 
	
	\section{Brauer pairs for $p$-permutation modules and $p$-permutation equivalences}
	
	Let $A$ and $B$ be sums of blocks of $kG$ and $kH$, respectively. Boltje and Perepelitsky in \cite{BoPe20} introduced the notion of a Brauer pair of a $p$-permutation equivalence $\gamma$ (to be precise, they initially defined the notion for an element $\gamma \in T^\Delta(A,B)$ for which $\gamma \cdot_B \gamma^* = [A]$). 
 
    Because a $p$-permutation equivalence can be written as a difference of bimodules whose indecomposable direct summands have twisted diagonal vertices, applying the Brauer construction at a subgroup $X \leq G\times H$ to a $p$-permutation equivalence will only be nonzero when $X$ is twisted diagonal. However, this is not a sufficient criteria for the result to be nonzero, and this motivates the definitions given in this section. We review the definitions and results in \cite{BoPe20}.

	\begin{definition}{\cite[Definition 5.1]{BoPe20}}
		Let $M \in {}_{kG}\textbf{triv}$. We call a Brauer pair $(P,e)$ of $kG$ an $M$-Brauer pair if $M(P,e) := e\cdot M(P) \neq 0$, where $M$ has $k[N_G(P,e)]e$-module structure. We denote the set of Brauer pairs of $M$ by $\mathcal{BP}(M)$. 
	\end{definition}

	\begin{prop}{\cite[Proposition 5.3, 5.4]{BoPe20}}\label{prop:M-Brauer pair properties} Let $M,N \in {}_{kG}\mathbf{triv}$.
		\begin{enumerate}
			\item $\mathcal{BP}(M)$ is a $G$-stable ideal in the poset $\mathcal{BP}(kG)$, i.e. it is stable under $G$-conjugation and if $(Q,f) \leq (P,e)$ are Brauer pairs of $kG$ such that $(P,e)$ is an $M$-Brauer pair, then also $(Q,f)$ is an $M$-Brauer pair. 
			\item Assume $M$ is indecomposable. Then the maximal $M$-Brauer pairs are precisely the $M$-Brauer pairs $(P,e)$ where $P$ is a vertex of $M$. Moreover, any two maximal $M$-Brauer pairs are $G$-conjugate. 
			\item Let $M$ and $N$ be indecomposable $p$-permutation $kG$-modules, suppose that $(P,e) \in \mathcal{BP}(kG)$ is both a maximal $M$-Brauer pair and a maximal $N$-Brauer pair, and set $I := N_G(P,e)$. Then $M(P,e)$ and $N(P,e)$ are indecomposable $p$-permutation $kIe$-modules. Moreover, $M \cong N$ if and only if $M(P,e) \cong N(P,e)$ as $kIe$-modules. 
		\end{enumerate}
		
	\end{prop}

    One may also define Brauer pairs for virtual $p$-permutation modules. 
 
	\begin{definition}{\cite[Definition 9.6]{BoPe20}}
		Let $\omega \in T(kG)$. A Brauer pair $(P,e)$ of $kG$ is called an $\omega$-Brauer pair if $\omega(P,e) = e\omega(P) \neq 0$ in $T(k[N_G(P,e)])$. The set of $\omega$-Brauer pairs is denoted $\mathcal{BP}(\omega)$. 
	\end{definition}
	
	\begin{remark}
		\begin{enumerate}
            \item If $\gamma \in T^\Delta(A,B)$ is a $p$-permutation equivalence for $A$ and $B$ and $\Delta(P,\phi,Q) \subset G\times H$ is a twisted diagonal subgroup, then via restriction, one can also view $\gamma(\Delta(P,\phi, Q)) \in T(k[N_{G\times H}(\Delta(P, \phi, Q))])$ as an element of the corresponding Grothendieck group of $C_G(P)\times C_H(Q) = C_{G\times H}(\Delta(P, \phi, Q))$.
			\item Let $\Delta(P,\phi,Q)$ be a twisted diagonal $p$-subgroup of $G\times H$. If $(P,e)$ is a Brauer pair of $kG$ and $(Q,f)$ is a Brauer pair of $kH$, then $(\Delta(P,\phi,Q),e\otimes f^*)$ is a Brauer pair of $k[G\times H]$, and conversely, if $(\Delta(P,\phi,Q), e\otimes f^*)$ is a Brauer pair of $k[G\times H]$, then $(P,e)$ is a Brauer pair of $kG$ and $(Q,f)$ is a Brauer pair of $kH$. In this case, with $I := N_G(P,e)$ and $J:= N_H(Q,f)$, one has $N_{G\times H}(\Delta(P,\phi, Q), e\otimes f^*) = N_{I\times J}(\Delta(P,\phi, Q))$ and $e\gamma(P,\phi, Q)f \in T^\Delta(k[N_{I\times J}(\Delta(P, \phi, Q))](e\otimes f^*))$. Moreover, $(\Delta(P,\phi, Q), e\otimes f^*)$ is an $A\otimes B^*$-Brauer pair if and only if $(P,e)$ is an $A$-Brauer pair and $(Q,f)$ is a $B$-Brauer pair. Denote the set of $A\otimes B^*$-Brauer pairs $(X,d)$ where $X\leq G\times H$ is a twisted diagonal subgroup by $\mathcal{BP}^\Delta(A,B).$
			\item Let $\Delta(P', \phi', Q') \leq \Delta(P,\phi, Q)$ be twisted diagonal subgroups of $G\times H$, and let $(\Delta(P', \phi', Q'), e'\otimes f'^*)$ and $(\Delta(P, \phi, Q), e\otimes f^*)$ be $k[G\times H]$-Brauer pairs. Then $(\Delta(P', \phi', Q'), e'\otimes f'^*) \leq (\Delta(P, \phi, Q), e\otimes f^*)$ if and only if $(P',e') \leq (P,e)$ and $(Q',f') \leq (Q,f)$. 
		\end{enumerate}
	\end{remark}
	
	The following proposition gives convenient reformulations of being a Brauer pair for a $p$-permutation equivalence.
	
	\begin{prop}{\cite[Proposition 10.8]{BoPe20}}
		Suppose $\gamma$ is a $p$-permutation equivalence for $A$ and $B$. Let $(\Delta(P,\phi, Q), e\otimes f^*) \in \mathcal{BP}^\Delta(A,B)$ and set $I := N_G(P,e)$ and $J := N_H(Q,f)$. The following are equivalent:
		\begin{enumerate}
			\item $(\Delta(P, \phi, Q), e\otimes f^*) \in \mathcal{BP}(\gamma)$, i.e. $e\gamma(P, \phi, Q)f \neq 0$ in $T(k[N_{I\times J}(\Delta(P, \phi, Q))](e\otimes f^*))$.
			\item $e \gamma(P,\phi, Q)f \neq 0$ in $T(k[C_G(P)]e, k[C_H(Q)]f)$.
			\item $(\Delta(P,\phi, Q), e\otimes f^*) \in \mathcal{BP}(M)$ for some indecomposable module $M \in {}_A\triv_B$ appearing in $\gamma$.
		\end{enumerate}
	\end{prop}

	\begin{prop}{\cite[Corollary 10.9]{BoPe20}}\label{uniqueness}
		Let $(P,e)$ be an $A$-Brauer pair and define $\Lambda$ as in Theorem \ref{7.5} and let $\Lambda_B \subseteq \Lambda$ be the set of pairs $(\psi,(Q,f))$ such that $(Q,f)$ is a $B$-Brauer pair. There exists a unique $H$-orbit of pairs $(\psi, (Q,f)) \in \Lambda_B$ with the property that $(\Delta(P,\psi, Q), e\otimes f^*)$ is a $\gamma$-Brauer pair. 
	\end{prop}

        \begin{prop}\cite[Theorem 10.10]{BoPe20}\label{prop:breaking up a p-perm equiv} Suppose $\gamma$ is a $p$-permutation equivalence for $A$ and $B$. Let $\mathcal{I}$ denote the set of primitive idempotents in $Z(A)$ and let $\mathcal{J}$ denote the set of primitive idempotents in $Z(B)$.
        \begin{enumerate}
         \item For every $e\in\mathcal{I}$, there exists a unique $f_e\in \mathcal{J}$ such that $e\gamma f_e \neq 0$. The assignment $e \mapsto f_e$ establishes a bijection $f:\mathcal{I}\to \mathcal{J}$. 
        
            \item We have $\gamma= \sum_{e\in \mathcal{I}}e\gamma f_e$. Moreover, for every $e \in \mathcal{I}$, $e\gamma f_e$ is a $p$-permutation equivalence for $kGe$ and $kGf_e$. 
            
            \item For every $e \in \mathcal{I}$, the $e\gamma f_e$-Brauer pairs are precisely the $\gamma$-Brauer pairs which contain $(1,e\otimes f_e^*)$. Thus, we have $\mathcal{BP}(\gamma) = \coprod_{e\in \mathcal{I}}\mathcal{BP}(e\gamma f_e)$. 
        \end{enumerate}
        \begin{proof}
            For claims (a) and (b), see \cite[Theorem 10.10]{BoPe20}. For (c), let $\lambda : = (\Delta(P,\phi,Q),b\otimes c^*)$ be a Brauer pair for $k[G\times H]$. Then
            \[
                (e\gamma f_e)(\lambda)= b(e\gamma f_e)(\Delta(P,\phi,Q))c = b\Br_P(e)\gamma(\Delta(P,\phi,Q))\Br_Q(f_e)c.
            \]
            The condition that $b\Br_P(e)\neq 0$ and $\Br_Q(f_e)c \neq 0$ is equivalent to $b\Br_P(e) = b$ and $\Br_Q(f_e)c = c$ and also to $(1,e\otimes f_e^*) \leq \lambda$. Considering the above equation, it follows immediately that $\lambda$ is a $e\gamma f_e$-Brauer pair if and only if $\lambda$ is a $\gamma$-Brauer pair and $(1,e\otimes f_e^*) \leq \lambda$. 
        \end{proof}
     
         \end{prop}
         
	\begin{theorem}{\cite[Theorem 10.11]{BoPe20}}\label{thm:maximal p-perm brauer pairs}
		\begin{enumerate}
			\item The set of $\gamma$-Brauer pairs forms a $G\times H$-stable ideal in the poset of $A\otimes B^*$-Brauer pairs.
			\item If $A$ and $B$ are blocks, then any two maximal $\gamma$-Brauer pairs are $G\times H$-conjugate.
			\item For $(\Delta(P,\phi, Q), e\otimes f^*)\in \mathcal{BP}(\gamma)$, the following are equivalent:
			\begin{enumerate}[label=(\roman*)]            
                \item $(\Delta(P,\phi, Q), e\otimes f^*)$ is a maximal $\gamma$-Brauer pair;
				\item $(P,e)$ is a maximal $A$-Brauer pair;
				\item $(Q,f)$ is a maximal $B$-Brauer pair.
			\end{enumerate}
		\end{enumerate}
	\end{theorem}

	Finally, it was proved that applying the Brauer construction to $p$-permutation equivalences at Brauer pairs induces local $p$-permutation equivalences on multiple levels.
	
	\begin{prop}{\cite[Proposition 11.1]{BoPe20}}\label{prop:iso determined by a brauer pair}
		Let $(\Delta(P,\phi, Q), e\otimes f^*)$ be a $\gamma$-Brauer pair and set $I:= N_G(P,e)$ and $J := N_H(Q,f)$. For every $g \in I$ there exists a unique element $hC_H(Q)\in J/C_H(Q)$ such that $c_g \circ \phi = \phi \circ c_h: Q \to P$. Similarly for every $h \in J$, there exists a unique element $gC_G(P) \in I/C_G(P)$ such that $c_g\circ \phi = \phi \circ c_h$. These associations define mutually inverse group isomorphisms between $I/C_G(P)$ and $J/C_H(Q)$. The isomorphism $J/C_H(Q) \cong I/C_G(P)$ restricts to an isomorphism $QC_H(Q)/Q \cong PC_G(P)/P, hC_H(Q) \mapsto \phi(h)C_G(P)$, for $h \in Q$. The group $Y:= N_{G\times H}(\Delta(P,\phi, Q), e\otimes f^*)$ satisfies $p_1(Y) = I, p_2(Y) = J, k_1(Y) = C_G(P), k_2(Y) = C_H(Q)$.
	\end{prop}

	\begin{theorem}{\cite[Theorem 11.4]{BoPe20}}\label{localpperms}
		Let $(\Delta(P,\phi, Q), e\otimes f^*) \in \mathcal{BP}(\gamma)$ be a $\gamma$-Brauer pair and set $I:= N_G(P,e)$, $J := N_H(Q,f)$. Suppose $C_G(P)\leq S\leq I$ and $C_H(Q) \leq T\leq J$ are intermediate groups related via the isomorphism in the previous proposition, and set $Y := N_{S\times H}(\Delta(P,\phi, Q), e\otimes f^*) = N_{S\times T}(\Delta(P,\phi, Q))$. Then the element $e\gamma(P,\phi,Q)f \in T(kY(e\otimes f^*))$ after restriction to $Y$ satisfies \[k[C_G(P)]e = e\gamma(P,\phi, Q)f \cdot^{Y,Y^\circ}_H f\gamma(P,\phi, Q)^* e \in T(k[N_{S\times S}(\Delta(P))](e\otimes e^*)).\] Moreover, the element \[\overline{\gamma} := \Ind_Y^{S\times T} (e\gamma(P,\phi, Q)f) \in T^\Delta(kSe, kTf)\] satisfies $\overline{\gamma}\cdot_T \overline{\gamma}^* = [kSe] \in T^\Delta(kSe, kSe)$. In particular, if $\gamma$ is a $p$-permutation equivalence between $A$ and $B$, then $\overline{\gamma}$ is a $p$-permutation equivalence between $kSe$ and $kTf$. 
	\end{theorem}
	
	\begin{corollary}{\cite[Proposition 11.5]{BoPe20}}\label{maximallocalbrauerpairs}
		Let $(\Delta(D,\phi, E), e_D\otimes f^*_E)$ be a maximal $\gamma$-Brauer pair. In particular, $D$ is a defect group of $A$ and $E$ is a defect group of $B$. For any subgroup $P\leq D$, let $(P,e_P)$ denote the unique $A$-Brauer pair with $(P,e_P) \leq (D,e_D)$ and for any subgroup $Q\leq E$ let $(Q,f_Q)$ denote the unique $B$-Brauer pair with $(Q,f_Q) \leq (E, f_E)$. Let $Q\leq E$ and set $P := \phi(Q)$. Then $(\Delta(P,\phi, Q), e_P\otimes f^*_Q)$ is a $\gamma$-Brauer pair. 
		\begin{enumerate}
			\item Set $\gamma' := e_P \gamma(P,\phi, Q)f_Q \in T^\Delta(k[C_G(P)]e_P, k[C_H(Q)]f_Q).$ $\gamma'$ is a $p$-permutation equivalence between $k[C_G(P)]e_P$ and $k[C_H(Q)]f_Q$. The $k[C_G(P)]e_P \otimes k[C_H(Q)]f^*_Q$-Brauer pair $(\Delta(C_D(P), \phi, C_E(Q)), e_{PC_D(P)}\otimes f^*_{QC_E(Q)})$ is a $\gamma'$-Brauer pair. It is a maximal $\gamma'$-Brauer pair if and only if $P$ is fully $\calA$-centralized if and only if $Q$ is fully $\calB$-centralized. 
			\item Set $I,J$ as before, and $\gamma'' := \Ind^{I\times J}_{N_{I\times J}(\Delta(P,\phi, Q)} (e_P\gamma(P,\phi, Q)f_Q) \in T^\Delta(kIe_P, kJf_Q)$. $\gamma''$ is a $p$-permutation equivalence between $kIe_P$ and $kJf_Q$. The $kIe_P \otimes kJf_Q^*$-Brauer pair \\$(\Delta(N_D(P), \phi, N_E(Q)), e_{N_D(P)}\otimes f^*_{N_H(Q)})$ is a $\gamma''$-Brauer pair. It is a maximal $\gamma''$-Brauer pair if and only if $P$ is fully $\calA$-normalized if and only if $Q$ is fully $\calB$-normalized. 
		\end{enumerate}
	\end{corollary}

	\section{Brauer pairs for splendid Rickard equivalences}
     
	Each splendid Rickard equivalence $C$ induces a $p$-permutation equivalence $\gamma := \Lambda(C)$, and therefore has an associated set of $\gamma$-Brauer pairs. As with $p$-permutation equivalences, in general applying the Brauer construction at an arbitrary subgroup to $C$ will not yield a local splendid Rickard equivalence. We may define a Brauer pair for a splendid Rickard equivalence in the same way as we have for $p$-permutation equivalences, treating contractible complexes as zero objects. In fact, the Lefschetz invariant of a contractible complex is zero in the corresponding Grothendieck group, but the converse of the statement rarely holds.
 
    We will show that applying the Brauer construction at the Brauer pairs for the corresponding $p$-permutation equivalence $\gamma$ do indeed yield local splendid equivalences. Moreover, we demonstrate that applying the Brauer construction at any non-Brauer pair results in a contractible complex. The main result of this paper is the corollary of the two above results, which is that the Brauer pairs for a splendid Rickard equivalence and its corresponding $p$-permutation equivalence are the same. 
	
	For this section, suppose $A, B$ are direct summands of group algebras $kG$ and $kH$ respectively. Let $C$ be a splendid Rickard equivalence inducing a splendid equivalence between $A$ and $B$, and set $\gamma := \Lambda(C)$. Given a Brauer pair $\omega = (\Delta(P,\phi, Q),e \otimes f^*)$, denote by $\omega^\circ$ the ``opposite'' Brauer pair $(\Delta(Q, \phi\inv, P), f \otimes e^*).$

	\begin{definition}
		Let $C$ be a bounded chain complex of $p$-permutation $(A,B)$-bimodules and set $\omega = (\Delta(P,\phi,Q), e\otimes f^*) \in \mathcal{BP}^\Delta(A,B)$. We say that $\omega$ \textit{is a Brauer pair for $C$} if and only if $C(\omega) := eC(\Delta(P,\phi, Q))f \not\simeq 0$ as a complex of $k[N_{I\times J}(\Delta(P,\phi,Q))](e\otimes f^*)$-modules where $I=N_G(P,e)$ and $J=N_G(Q,f)$. Denote the set of $C$-Brauer pairs by $\mathcal{BP}(C)$. 
	\end{definition}
	
	Let $C$ be a splendid Rickard equivalence. The goal of this section is to determine the set $\calB\calP(C)$ of $C$-Brauer pairs. Notice that the Lefschetz invariant of $C(\omega)$ is $\gamma(\omega)$. Since the Lefschetz invariant of a contractible complex is zero, it is immediate that $\mathcal{BP}(\gamma) \subseteq \mathcal{BP}(C)$. However, a noncontractible complex can have Lefschetz invariant equal to zero, so it is not clear whether or not these sets are equal. 
 
 We will make liberal use of the following standard homological algebra facts.
	
	\begin{prop}\label{prop:functors extended to chain complexes}
        \begin{enumerate}
            \item Let $\mathcal{A}, \mathcal{B}$ be abelian categories, and let $F,G :\mathcal{A} \to \mathcal{B}$ be naturally isomorphic additive functors. The induced functors $\overline{F}, \overline{G}: \Ch(\mathcal{A}) \to \Ch(\mathcal{B})$ are naturally isomorphic. 
            \item Let $\calA, \calB, \calC$ be abelian categories, and let $F, G: \calA \times \calB \to \calC$ be naturally isomorphic bifunctors which are additive in each component. The induced functors $\overline{F}, \overline{G}: Ch^b(\calA) \times Ch^b(\calB) \to Bi^b(\calC)$ are naturally isomorphic, where $Bi^b(\calC)$ denotes the category of bounded $\calC$-bicomplexes. Postcomposing with the functor corresponding to taking the total complex yields a natural isomorphism $\overline{F}, \overline{G}: Ch^b(\calA)\times Ch^b(\calB) \to \Ch^b(\calC)$.
        \end{enumerate}    
		
	\end{prop}
	
	In particular, the isomorphisms given in Theorem \ref{7.5} and Theorem \ref{bouciso} extend to chain complex isomorphisms. 
 
    First, we wish to reduce to the case that $A$ and $B$ are blocks. In order to do this, we will establish an analogue of Proposition~\ref{prop:breaking up a p-perm equiv} for splendid Rickard equivalences. The main tool is the following result. 

    \begin{lemma}\label{lem: orbit determined by A Brauer pair} Suppose that $C$ is a splendid Rickard equivalence between sums of blocks $A$ and $B$ and that $(P,e)$ is an $A$-Brauer pair. Let $\Lambda_B \subseteq \Lambda$ be the set of pairs $(\phi,(Q,f))$ where $(Q,f)$ is a $B$-Brauer pair and $\phi:Q \to P$ is a group isomorphism. There exists a unique $H$-orbit of pairs $(\phi,(Q,f)) \in \Lambda_B$ with the property that $eC(\Delta(P,\phi,Q))f$ is not contractible as a complex of $(k[C_G(P)]e, k[C_H(Q)]f)$-bimodules.

    \begin{proof}
        Since $C$ is a splendid Rickard equivalence, we have a homotopy equivalence $C \otimes_{B} C^* \simeq A$ of complexes of $(A,A)$-bimodules. We apply the functor $-(\Delta(P),e\otimes e^*)$ to this equation. Using Proposition~\ref{brauerconstructionforbrauerpairsprops}, Theorem~\ref{7.5}(d), and Proposition~\ref{brauercommuteswithdualsandconj}, we obtain an isomorphism
        	\begin{align*}
        		k[C_G(P)]e & \cong A(\Delta(P),e\otimes e^*)\\
        		&  \simeq \left( C \otimes_{kH} C^* \right)(\Delta(P),e\otimes e^*)\\
        		& \cong \bigoplus_{(\phi, (Q,f)) \in \widetilde\Lambda} e C(\Delta(P,\phi,Q))f \otimes_{k[C_H(Q)]} f C^*(\Delta(Q,\phi^{-1},P))e\\
        		& \cong \bigoplus_{(\phi, (Q,f)) \in \widetilde\Lambda \cap \Lambda_B} e C(\Delta(P,\phi,Q))f \otimes_{k[C_H(Q)]}  (eC(\Delta(P,\phi,Q))f)^*.
        	\end{align*}
         in the homotopy category of complexes of $(k[C_G(P)]e,k[C_G(P)]e)$-bimodules. The change in index is justified since $e C (P,\phi,Q) f = 0$ whenever $(Q,f)$ is not a $B$-Brauer pair. Since $k[C_G(P)]e$ is an indecomposable bimodule, the corresponding complex concentrated in degree zero is indecomposable when viewed as an object of the homotopy category. Using the Krull-Schmidt theorem, we conclude that there is a unique pair $(\phi, (Q,f)) \in \widetilde\Lambda \cap \Lambda_B$ such that 
         \[
         e C(\Delta(P,\phi,Q))f \otimes_{k[C_H(Q)]} (eC(\Delta(P,\phi,Q))f)^* \simeq  k[C_G(P)]e.
         \]
         It follows from Lemma~\ref{lem:contractible tensor product of complexes} that $e C(\Delta(P,\phi,Q))f \not\simeq 0$ as a complex of $(k[C_G(P)]e, k[C_H(Q)]f)$-bimodules. This proves the existence part of the statement. In order to prove uniqueness, suppose that $(\phi', (Q',f')) \in \widetilde\Lambda \cap \Lambda_B$ is any other pair. Set $C': = e C(\Delta(P,\phi',Q'))f'$ so that we have $C' \otimes_{k[C_H(Q')]} C'^* \simeq 0$. Thus, $C'$ is a complex of $(k[C_G(P)]e,k[C_H(Q')]f')$-bimodules satisfying the hypothesis of Lemma~\ref{lem:contractible tensor product of complexes}. From this, it follows that $e C(\Delta(P,\phi',Q'))f' \simeq 0$, as desired. 
        \end{proof} 
        \end{lemma}
 
    The next result shows that a splendid Rickard equivalence $C$ determines a bijection between the block direct summands of $A$ and $B$. Moreover, $C$ is a direct sum (in the homotopy category) of splendid Rickard equivalences for each pair of corresponding blocks and the set of Brauer pairs for $C$ breaks up as the disjoint union of the sets of Brauer pairs for each summand. 

        \begin{theorem}\label{thm:breaking up a splendid complex}Suppose $A$ and $B$ are sums of blocks and that $C$ is a splendid Rickard equivalence for $A$ and $B$. Let $f:\mathcal{I}\to\mathcal{J}$ be the bijection determined by $\gamma := \Lambda(C)$ from Proposition~\ref{prop:breaking up a p-perm equiv}. 
    \begin{enumerate}
        \item For every $e\in\mathcal{I}$, the element $f_e\in \mathcal{J}$ is unique with the property that $eCf_e\not\simeq 0$.
        
        \item We have $C= \bigoplus_{e\in \mathcal{I}}eCf_e$ in $K^b({}_{A}\mathbf{mod}_B)$. Moreover, for every $e \in \mathcal{I}$, $eCf_e$ is a splendid Rickard equivalence for $kGe$ and $kGf_e$. 
        
        \item For every $e \in \mathcal{I}$, the $eCf_e$-Brauer pairs are precisely the $C$-Brauer pairs which lie above $(1,e\otimes f_e^*)$. We have $\mathcal{BP}(C) = \coprod_{e\in \mathcal{I}}\mathcal{BP}(eCf_e)$. 
    \end{enumerate} 
    \begin{proof}The Lefchetz invariant of $eCf_e$ is $e\gamma f_e$, which is nonzero by Proposition~\ref{prop:breaking up a p-perm equiv}(a). Therefore, $eCf_e \not\simeq 0$. By applying Lemma~\ref{lem: orbit determined by A Brauer pair} to the splendid Rickard equivalence $C$ and the $A$-Brauer pair $(1,e)$, we obtain that there exists a unique $f \in \mathcal{J}$ such that $eCf \not\simeq 0$. By uniqueness, $f = f_e$. This proves (a). 
    
    Using (a), we have a direct sum decomposition
    \[
    C = \bigoplus_{e\in \mathcal{I}}eCf_e
    \]
    in the homotopy category of chain complexes of $(A,A)$-bimodules. Since $kGe$ and $kHf_e$ are blocks, it suffices by Proposition~\ref{prop:src left implies right} to show just one of the homotopy equivalences. Multiplying the above equation by $e$ yields $eC = eCf_e$. Thus, we obtain
    \begin{align*}
        kGe = eAe & \simeq e(C\otimes_B C^*)e = eCf_e \otimes_{kGf_e} (eCf_e)^*.
    \end{align*}
    This completes the proof of (b). Finally, let $\lambda =(\Delta(P,\phi,Q),b\otimes c^*)$ be a $k[G\times H]$-Brauer pair. Then using Lemma~\ref{lem:braur construction for iM}, we have
    \[
    (eCf_e)(\lambda)  = b(eCf_e)(\Delta(P,\phi,Q))c= b\Br_P(e)C(\Delta(P,\phi,Q))\Br_Q(f_e)c.
    \]
    As in the proof of Proposition~\ref{prop:breaking up a p-perm equiv}, it is easily deduced from this equation that $\lambda$ is a $eCf_e$-Brauer pair if and only if $\lambda$ is a $C$-Brauer pair satisfying $(1,e\otimes f_e^*)\leq \lambda$. Since every $C$-Brauer pair of $k[G\times H]$ lies above a unique block, we conclude that the union is disjoint. 
    \end{proof}
        
    \end{theorem}

    Therefore, it suffices to consider only splendid Rickard equivalences for block algebras. We remark that part (a) similarly holds for derived Rickard equivalences and stable equivalences of Morita type. Indeed, one may verify that a chain complex of bimodules inducing a derived equivalence between sums of block algebras decomposes into a direct sum of indecomposable derived Rickard complexes inducing equivalences of block algebras, and similarly for stable equivalences of Morita type.

    From here, assume $A$ and $B$ are block algebras. We next prove a chain complex-theoretic version of Theorem \ref{localpperms}.

	\begin{theorem}\label{thm:localequivalences}
		Suppose that $A$ and $B$ are blocks and $C$ is a splendid Rickard equivalence for $A$ and $B$, and set $\gamma := \Lambda(C)$. Set $\lambda := (\Delta(P,\phi, Q), e\otimes f^*)$ be a $\gamma$-Brauer pair and set $I := N_G(P,e), J := N_H(Q,f)$. Suppose $C_G(P) \leq S \leq I$ and $C_H(Q) \leq T \leq J$ are intermediate groups related via the isomorphism $I/C_G(P) \to J/C_H(Q)$ defined in Proposition \ref{prop:iso determined by a brauer pair}, and set $Y := N_{S\times H}(\Delta(P,\phi,Q),e\otimes f^*)$. Then the chain complex $C(\lambda) := eC(\Delta(P,\phi, Q))f \in \mathrm{Ch}^b({}_{kY(e\otimes f^*)}\triv)$ satisfies 
		\[
            k[C_G(P)]e \simeq C(\lambda) \mathop{\otimes}\limits_{kH}^{Y, Y^\circ} C(\lambda)^* \text{ in } K^b({}_{k[N_{S\times S}(\Delta(P))](e\otimes e^*)}\triv)
         \]
         and
          \[
          k[C_H(Q)]f \simeq C(\lambda)^* \mathop{\otimes}\limits_{kG}^{Y^\circ, Y} C(\lambda) \text{ in } K^b({}_{k[N_{T\times T}(\Delta(Q))](f\otimes f^*)}\triv).
          \] 
   In particular, $C(\lambda)$ is a splendid Rickard equivalence for $k[C_G(P)]e$ and $k[C_H(Q)]f$, and the chain complex \[\Ind_{Y}^{S\times T}(C(\lambda)) \in \mathrm{Ch}^b({}_{kSe}\triv_{kTf})\] is a splendid Rickard equivalence for $kSe$ and $kTf$. 
	\end{theorem}
	
	\begin{proof}
		We show that $k[C_G(P)]e \simeq C'\mathop{\otimes}\limits_{kH}^{Y,Y^\circ}C'^*$, and the dual homotopy equivalence follows similarly. 
  
        First, observe that the Lefchetz invariant of the chain complex $C(\lambda) \mathop{\otimes}\limits_{kH}^{Y, Y^\circ} C(\lambda)^*$ is equal to $\gamma(\lambda) \mathop{\cdot}\limits_{H}^{Y, Y^\circ} \gamma(\lambda)^*$. Since $\lambda \in \mathcal{BP}(\gamma)$, Theorem \ref{localpperms} implies that $\gamma(\lambda) \mathop{\cdot}\limits_{H}^{Y, Y^\circ} \gamma(\lambda)^*$ is nonzero. This implies that $C(\lambda) \mathop{\otimes}\limits_{kH}^{Y, Y^\circ} C(\lambda)^*$ is not contractible as a complex of $k[N_{S\times S}(\Delta(P)](e\otimes e^*)$-modules. 
        
        Next, we apply the Brauer construction at $\Delta(P)\leq G\times G $ to the homotopy equivalence of $(A,A)$-bimodules $C\otimes_{B}C^* \simeq A$. Using Lemma~\ref{lem:braur construction for iM}, we obtain a homotopy equivalence 
		\[kC_G(P)e \simeq e[C\otimes_BC^*](\Delta(P))e,\] 
		 of chain complexes of $k[N_{S\times S}(\Delta(P))](e\otimes e^*)$-modules. Using the notation in Theorem \ref{7.5}(d), choose a set of $S\times H$ orbit representatives $\widehat{\Lambda}\subseteq \Lambda$ which contains $(\phi,(Q,f))$. Using Theorem \ref{7.5} together with Proposition \ref{prop:functors extended to chain complexes} we obtain a homotopy equivalence  
		\begin{align*}k[C_G(P)]e & \simeq e(C\otimes_B C^*)(\Delta(P))e\\
        % & \cong \bigoplus_{ (\psi,(R,b)) \in \hat{\Lambda}} eC(\Delta(P,\psi, R))b\mathop{\otimes}\limits_{H}^{Y, Y^\circ} bC^*(\Delta(R,\psi\inv, P))e\\
        & \cong \bigoplus_{\tau = (\psi, (R,b))\in \widehat{\Lambda}} \Ind^{N_{S\times S}(\Delta(P))}_{\Delta(I(\tau))(C_G(P)\times 1)} (C(\tau) \mathop{\otimes}\limits_{kH}^{X(\tau),X(\tau)^\circ} C(\tau)^*)
        \end{align*} 
        of chain complexes of $k[N_{S\times S}(\Delta(P))](e\otimes e^*)$-modules. Notice that when $\tau = (\phi,(Q,f))$, we have $X(\tau) = Y$, $C(\tau) = C(\lambda)$, and, by Proposition \ref{prop:iso determined by a brauer pair}, $I(\tau):= N_{(S, \phi, J)} = S$. It follows from the latter that
        %Note that $X(\lambda) = Y$ and $\Delta(I(\lambda))(C_G(P)\times 1) = Y\ast Y^\circ$. A quick calculation shows that that $Y = N_{S\times J}(\Delta(P,\phi,Q))$. Using Lemma \ref{} together with the fact that $S$ and $T$ correspond under the isomorphism $I/C_G(P) \to J/C_H(Q)$, we see that $p_1(Y) = N_{(S,\phi,J)} = N_{(S,\phi,T)} = S$ and $p_2(Y) = N_{(J,\phi^{-1},S)} = N_{(T,\phi^{-1},S)} = T$. Now, Lemma \ref{} implies that 
         \[
           \Delta(I(\tau))(C_G(P)\times 1)= \Delta(S)(C_G(P)\times 1)= N_{S\times S}(\Delta(P)).
        \]
          Thus, $C(\lambda) \otimes_{H}^{Y, Y^\circ} C(\lambda)^*$ is a noncontractible direct summand of the complex on the right hand side.
	   %where as in \ref{7.5}, $\Lambda$ consists of pairs $(\psi,(R,b))$, with $\psi: R\to P$ is an isomorphism, $(R,b)$ a $B$-Brauer pair, and with $S \times H$-action given by ${}^{(g,h)}(\phi, (R,b)) = (c_g\phi c_h\inv, {}^h(R,b))$, for $(g,h) \in S\times H$. $\widehat{\Lambda} \subseteq \Lambda$ denotes a set of $S \times H$-orbit representatives. 
		On the other hand, $k[C_G(P)]e$ is an indecomposable $k[N_{S\times S}(\Delta(P))]$-bimodule, hence indecomposable and noncontractible as a chain complex. By the Krull-Schmidt theorem, precisely one of the terms in the direct sum must be noncontractible. We conclude that $k[C_G(P)]e \simeq C(\lambda) \otimes_{H}^{Y, Y^\circ} C(\lambda)^*$.
        %In particular, only one of the summands can have a nonzero Lefschetz invariant. By the theory of $\gamma$-Brauer pairs, this occurs by definition if and only if $(\psi,(R,b)) \in \widehat{\Lambda}$ corresponds to a $\gamma$-Brauer pair, i.e. $e\gamma(\Delta(P,\psi,R))b \neq 0$. Moreover, such an $\omega \in \hat{\Lambda}$ must occur since if otherwise, the Lefschetz invariant of the chain complex on the right-hand side would be zero, a contradiction. For the $\gamma$-Brauer pair for which this occurs, by \ref{uniqueness}, we can assume $(\psi, (R,b)) = (\phi, (Q,f))$, hence $(\Delta(P, \psi, R), e\otimes b^*) = \lambda$. Therefore, we have 
		%\[k[C_G(P)]e \simeq eC(\Delta(P,\psi, R))b\otimes_{H}^{Y, Y^\circ} bC^*(\Delta(R,\psi\inv, P))e = C(\lambda)\otimes_H^{Y,Y^\circ} C^*(\lambda^\circ) =C(\lambda)\otimes_H^{Y,Y^\circ} C(\lambda)^*, \] 
		%and for all other pairs $(\phi, (Q,f)) \neq \omega' \in \widehat{\Lambda}$, $0 \simeq C(\omega')\otimes_H^{Y,Y^\circ} C^*(\omega')$. The first homotopy equivalence follows. 
		
		Therefore, \[\Ind^{S\times S}_{N_{S\times S}(\Delta(P))}(k[C_G(P)]e) \simeq   \Ind^{S\times S}_{N_{S\times S}(\Delta(P))}(C(\lambda)\otimes_{kH}^{Y,Y^\circ} C^*(\lambda^\circ))\] 
		Observe $kC_G(P)$ is a transitive $k[N_{S\times S}(\Delta(P))]$-permutation module with $\Delta(S)$ stabilizing $1 \in C_G(P)$, hence 
		\[kC_G(P) \cong \Ind_{\Delta(S)}^{N_{S\times S}(\Delta(P))}(k),\] 
		and therefore 
		\[\Ind^{S\times S}_{N_{S\times S}(\Delta(P))} (k[C_G(P)]e)\cong e( \Ind^{S\times S}_{N_{S\times S}(\Delta(P))} (kC_G(P)))e \cong e\Ind^{S\times S}_{\Delta(S)}(k) e \cong kS e.\] 
		So we have 
		\[
             \Ind^{S\times S}_{N_{S\times S}(\Delta(P))} (C(\lambda) \otimes_{kH}^{Y,Y^\circ}C^*(\lambda^\circ))\simeq\Ind^{S\times S}_{N_{S\times S}(\Delta(P))} (k[C_G(P)]e)\cong kSe.
             \] 
		We show that the left-hand side is isomorphic to 
        \[
        \Ind_Y^{S\times T}(C(\lambda)) \otimes_{kT} (\Ind^{S\times T}_{Y} C(\lambda))^*.
        \] 
	   First, notice that $Y= N_{S\times J}(\Delta(P,\phi,Q))$ using the definitions. Then Lemma \ref{prop:proj&kernelsforspecialsubgroups}(c) implies that $p_1(Y) = N_{(S,\phi, J)}$ and $p_2(Y) = N_{(J,\phi^{-1}, S)}$. Since $\lambda \in \mathcal{BP}(\gamma)$ and $S$ and $T$ are related under the isomorphism $I/C_G(P)\to J/C_G(Q)$, Proposition \ref{prop:iso determined by a brauer pair} implies that $p_1(Y) = N_{(S,\phi, J)} = S$ and $p_2(Y) = N_{(J,\phi^{-1}, S)} = T$. In particular, $Y\leq S\times T $ so the claim makes sense. 
        
        Next, we have the identifications \[(\Ind_Y^{S\times T} C(\lambda))^* = (\Ind_Y^{S\times T}(eC(P,\phi, Q)f))^* \cong \Ind_{Y^\circ}^{T\times S}(fC^*(Q, \phi\inv, P) e).\] 
		Then applying Theorem \ref{bouciso}, we have
		\[\Ind^{S\times T}_{Y}(C(\lambda))\otimes_{kT} \Ind^{T\times S}_{Y^{o}}(C^*(\lambda^\circ)) \cong \bigoplus_{t \in [p_2(Y)\backslash T/p_1(Y^\circ)]} \Ind^{S\times S}_{Y\ast {}^{(t,1)}Y^\circ} \big(eC(P,\phi, Q)f\otimes_{kT}^{Y, {}^{(t,1)}Y} fC^*(Q,\phi\inv, P)e\big).\] 
		However, $p_2(Y) = p_1(Y^\circ) = T$, therefore the indexing set in the direct sum has cardinality one, and the corresponding term is $\Ind^{S\times S}_{Y} \big( C(\lambda) \otimes_{kT}^{Y, Y^\circ} C^*(\lambda^\circ)\big)$. We conclude that \[kSe \simeq \Ind^{S\times S}_{Y}\left( C(\lambda) \otimes_{kT}^{Y, Y^\circ} C^*(\lambda^\circ)\right) \cong \Ind^{S\times T}_Y(C(\lambda)) \otimes_{kT} (\Ind_Y^{S\times T}(C(\lambda)))^*  .\] 
		The second homotopy equivalence follows similarly, and therefore $\Ind^{S\times T}_{Y}(C(\lambda))$ is a splendid Rickard equivalence for $kSe$ and $kTf$. 
	\end{proof}
	
	We next show that in addition, if $\omega := (\Delta(P, \phi, Q), e\otimes f^*)$ is not a $\gamma$-Brauer pair, then $C(\omega) \simeq 0$ as a complex of $k[N_{G\times H}(\omega)]$-modules. First, we make some easy observations. These are well-known, but we include proofs for convenience. 
	
	\begin{lemma}\label{lem:contractible induction}
		\begin{enumerate}
			\item Let $C$ be a complex of $(A,B)$-bimodules and let $C'$ be a complex of $(B,B')$-bimodules for some $k$-algebras $A,B,B'$. If $C^* \otimes_A C \simeq B$ and $C\otimes_BC'\simeq 0$, then $C' \simeq 0$. 
			\item Let $H\leq G$ and let $C$ be a chain complex of $kH$-modules. Then $\Ind^G_H C \simeq 0$ if and only if $C \simeq 0$. 
		\end{enumerate}
	\end{lemma}
	\begin{proof}
		\begin{enumerate}
			\item We have \[C' \cong B \otimes_B C' \simeq (C^*\otimes_A C) \otimes_B C' = C^*\otimes_A (C \otimes_B C') \simeq C^* \otimes_A 0 \cong 0.\].
			\item The Mackey formula is natural and therefore extends to chain complexes. We have \[ \Res^G_H \Ind^G_H C \cong \bigoplus_{g \in [H\backslash G / H]} \Ind^H_{H\cap {}^gH} \Res^H_{H\cap {}^gH} C.\] Suppose $\Ind^G_HC \simeq 0$. Then the left-hand side is contractible since restriction preserves contractibility. All summands on the right-hand side are contractible as well, since direct summands of contractible chain complexes are again contractible. Since there is a coset representative in $[H\backslash G/H]$ belonging to $H$, $C$ is isomorphic to a direct summand of $\Res^G_H\Ind^G_H C$, and thus is contractible. The reverse direction is clear. 
		\end{enumerate}
	\end{proof}
	
	\begin{theorem}\label{thm:not a braur pair->contractible.}
		Suppose $C$ is a splendid Rickard equivalence for $A$ and $B$, with $\gamma = \Lambda(C)$. 
        %Let $(D,e_D)$ be a maximal Brauer pair of $A$ and $(E,e_E)$ be a maximal Brauer pair of $B$. 
        If $\omega:= (\Delta(R, \psi, Q), d\otimes f^*)\in \mathcal{BP}^\Delta(A,B)$ is not a $\gamma$-Brauer pair, then $C(\omega) := dC(\Delta(R,\psi, Q))f$ is contractible as a chain complex of $k[N_{G\times H}(\Delta(R,\psi, Q), d\otimes f^*)](d\otimes f^*)$-modules. In particular, $\omega$ is not a $C$-Brauer pair. 
	\end{theorem}
	\begin{proof}%-----Revised Proof------%
         There exists an $A$-Brauer pair $(P, e)$ and an isomorphism $\phi:Q\to P$ such that $\lambda := (\Delta(P,\phi, Q), e\otimes f^*)$ is a $\gamma$-Brauer pair. Set $I = N_G(P,e), \ J = N_H(Q,f)$, $K = N_G(R,d)$, $X = N_{G\times H}(\lambda)$, and $Y = N_{G\times H}(\omega)^\circ$. By Theorem \ref{thm:localequivalences}, $\Ind_{X}^{I\times J}(C(\lambda))$ is a splendid Rickard equivalence for $kIe$ and $kJf$. In order to prove the claim, it suffices to show that 
        \[
            \Ind_{X}^{I\times J}(C(\lambda))\mathop{\otimes}\limits_{kJ}\Ind_{Y}^{J\times K}(C(\omega)^*)\simeq 0.
        \]
        If so, Lemma \ref{lem:contractible induction}(a) implies that $\Ind_{Y}^{J\times K}(C(\omega)^*)\simeq 0$. Then Lemma \ref{lem:contractible induction}(b) implies that $C(\omega)^*\simeq 0$. By taking duals, we obtain $C(\omega)\simeq 0$, as desired.  
        
        In order to prove the above claim, we apply Theorem \ref{bouciso} to obtain
        \[
            \Ind_{X}^{I\times J}(C(\lambda))\mathop{\otimes}\limits_{kJ}\Ind_{Y}^{J\times K}(C(\omega)^*) \cong \bigoplus_{t \in [ p_2(X)\backslash J/ p_1(Y)]} \Ind_{X\ast{}^{(t,1)}Y}^{I\times K}\left(C(\lambda) \mathop{\otimes}\limits_{kJ}^{X,{}^{(t,1)}Y} {}^{(t,1)}(C(\omega)^*)\right).
        \]
        Thus, it suffices to show that every complex on the right hand side is contractible. Let $t \in J$. Set $\omega_t = {}^{(1,t)}\omega $. Notice that $\omega_t^\circ = (\Delta(Q,c_t\circ \psi^{-1},R),f\otimes d^*)$ since $t\in J$.
        For ease of notation, write 
        \[
        \lambda \circ\omega_t^\circ:= (\Delta(P, \phi\circ c_t \circ\psi\inv, R), e\otimes d^*).
        \]
        Let $\Omega$ denote the set of triples $(\sigma, (S, a), \tau)$ where $(S,a)$ is a $kH$-Brauer pair and $\tau: R\to S$ and $\sigma: S \to P$ are isomorphisms for which $\sigma\circ\tau = \phi\circ c_t \circ \psi\inv$. Choose a set $\hat{\Omega}$ of $N_{I\times K}(\Delta(P, \phi\circ c_t \circ \psi\inv, R)\times H$-orbit representatives of $\Omega$ which contains $(\phi,(Q,f),c_t \circ \psi\inv)$. Then Theorem \ref{7.5}(c) gives an isomorphism 
		\[
        \bigoplus_{\nu \in \hat{\Omega}} \Ind_{X(\nu) \ast Y(\nu)}^{N_{I\times K}(\Delta(P, \phi\circ c_t\circ \psi\inv, R))} \left(C(\nu) \mathop{\otimes}\limits_{kH}^{X(\nu), Y(\nu)}C^*(\nu)\right)\cong (C\otimes_{kH}C^*)(\lambda\circ\omega_t^\circ),
        \] 
		where for $\nu = (\sigma, (S,a), \tau))$, we set $X(\nu) := N_{I\times H}(\Delta(P,\sigma, S),e\otimes a^*)$, $Y(\nu) := N_{H\times K}(\Delta(S, \tau, R), a\otimes d^*)$, $C(\nu) := eC(\Delta(P, \sigma, S))a$ and $C^*(\nu) := aC^*(\Delta(S,\tau, R))d$. In particular, when $\nu = (\phi, (Q, f), c_t \circ \psi\inv)$, we have $X(\nu) = X$, $Y(\nu) = {}^{(t,1)}Y$, $C(\nu) = C(\lambda)$ and, by Proposition \ref{brauercommuteswithdualsandconj}, $C^*(\nu) \cong {}^{(t,1)}(C(\omega)^*)$. Thus, since $(\phi, (Q, f), c_t \circ \psi\inv) \in \hat{\Omega}$, we conclude that
        \[
        \Ind_{X \ast {}^{(t,1)}Y}^{N_{I\times K}(\Delta(P, \phi\circ c_t\circ \psi\inv, R))} \left(C(\lambda) \mathop{\otimes}_{kH}^{X, {}^{(t,1)}Y}{}^{(t,1)}(C(\omega)^*)\right)
        \] 
        is homotopy equivalent to a direct summand of 
        \[
        (C\otimes_{kH}C^*)(\lambda\circ\omega^\circ) \simeq A(\lambda\circ\omega_t^\circ).
        \]
		
		We claim $A(\lambda\circ\omega_t^\circ) = eA(\Delta(P, \phi\circ c_t \circ \psi\inv, R))d = 0$. Suppose to the contrary that $A(\lambda\circ\omega_t^\circ) \neq 0$. Then $\lambda\circ\omega_t^\circ$ is a Brauer pair for $A$. However, the Brauer pairs of any block algebra (regarded as bimodule) are up to $G\times G$-conjugacy given by pairs of the form $(\Delta(S), e_S\otimes e_S^*)$ for some $S \leq G$ and some block $e_S$. So there exist $(g_1, g_2) \in G\times G$ for which 
        \[
        (\Delta(P, \phi\circ c_t \circ \psi\inv, R), e\otimes d^*) = {}^{(g_1,g_2)}(\Delta(S), e_S\otimes e_S^*) = (\Delta({}^{g_1}S, c_{g_1g_2\inv},{}^{g_2}S), {}^{g_1}e_S \otimes {}^{g_2}e_S^*).
        \] 
        It is a straightforward computation that \[{}^{g_1}e_S = e, {}^{g_2}e_S = d, P = {}^{g_1}S, R = {}^{g_2}S, \phi \circ c_t = c_{g_1g_2\inv}\circ \psi.\] In this case,
		\begin{align*}
		{}^{(g_2g_1\inv, 1)}\lambda &=  {}^{(g_2g_1\inv, 1)}(\Delta(P, \phi, Q), e\otimes f^*)\\
		&= (\Delta({}^{g_2g_1\inv}P, c_{g_2g_1\inv}\circ \phi, Q), {}^{g_2g_1\inv}e\otimes f^*)\\
		&= (\Delta(R, \psi \circ c_{t^{-1}}, Q), d\otimes f^*) = \omega_t.
		\end{align*}
		This implies that $\omega_t$ is a $\gamma$-Brauer pair since it is conjugate to the $\gamma$-Brauer pair $\lambda$. However, $\omega$ and $\omega_t$ are also conjugate, so we conclude that $\omega$ is a $\gamma$-Brauer pair, a contradiction. Thus $A(\lambda\circ\omega_t^\circ)= 0$. Since direct summands of contractible complexes are contractible, we obtain
         \[
            \Ind_{X\ast{}^{(t,1)}Y}^{N_{I\times K}(\Delta(P, \phi\circ c_t\circ \psi\inv, R))}\left(C(\lambda) \mathop{\otimes}\limits_{kJ}^{X,{}^{(t,1)}Y} {}^{(t,1)}(C(\omega)^*)\right) \simeq 0.
         \]
         Finally, by applying the functor $\Ind_{N_{I\times K}(\Delta(P, \phi\circ c_t\circ \psi\inv, R))}^{I\times K}(-)$ to this equation, we obtain
        \[
        \Ind_{X\ast{}^{(t,1)}Y}^{I\times K}\left(C(\lambda) \mathop{\otimes}\limits_{kJ}^{X,{}^{(t,1)}Y} {}^{(t,1)}(C(\omega)^*)\right) \simeq 0
        \] 
        for all $t \in J$. This completes the proof.

	\end{proof}
	
	As a corollary of the preceding theorem, we obtain that the Brauer pairs for a splendid Rickard equivalence and the Brauer pairs for its Lefschetz invariant coincide. This allows us to reinterpret the results of \cite{BoPe20} to the setting of splendid equivalences.
	
	\begin{corollary}{c.f. \cite[Theorems 10.11, 11.2]{BoPe20}}\label{mainresult}
		Let $C$ be a splendid Rickard equivalence between $A$ and $B$ and let $\gamma:= \Lambda(C)$ be its corresponding $p$-permutation equivalence. Then $\mathcal{BP}(C) = \mathcal{BP}(\gamma)$. In particular, 
  \begin{enumerate}
    \item $\mathcal{BP}(C)$ is a $G\times H$-stable ideal in the poset of $A\otimes B^*$-Brauer pairs. 
      \item If $A$ and $B$ are blocks, then any two maximal $C$-Brauer pairs are $G\times H$-conjugate.
      \item The following are equivalent for $(\Delta(P,\phi,Q),e\otimes f^*) \in \calB\calP(C)$:
      \begin{enumerate}[label=(\roman*)]
          \item $(\Delta(P,\phi,Q),e\otimes f^*)$ is a maximal $C$-Brauer pair.
          \item $(P,e)$ is a maximal $A$-Brauer pair.
          \item $(Q,f)$ is a maximal $B$-Brauer pair.
      \end{enumerate}
      \item Suppose $A$ and $B$ are blocks and $(\Delta(D,\phi,E),e_D\otimes f_E^*)$ is a maximal $C$-Brauer pair. Let $\mathcal{A}$ denote the fusion system of $A$ associated to $(D,e_D)$ and let $\mathcal{B}$ denote the fusion system of $B$ associated to $(E,f_E)$.  Then $\phi:E \to D$ is an isomorphism of the fusion systems $\mathcal{B}$ and $\mathcal{A}$.
    \end{enumerate}
    \begin{proof}
        By Proposition~\ref{prop:breaking up a p-perm equiv} and Theorem~\ref{thm:breaking up a splendid complex}, we may assume that $A$ and $B$ are blocks. Then the fact that $\mathcal{BP}(C) \subseteq \mathcal{BP}(\gamma)$ follows from Theorem \ref{thm:not a braur pair->contractible.}. The other inclusion $\mathcal{BP}(\gamma)\subseteq \mathcal{BP}(C)$ was already established at the beginning of this section. Claims (a)-(c) above follow immediately from \cite[Theorem 10.11]{BoPe20}. Claim (d) follows from \cite[Theorem 11.2]{BoPe20}.
  \end{proof}
        % Let $\omega = (\Delta(P,\phi,Q), e\otimes f^*) \in \mathcal{BP}^\Delta(A ,B)$. Then $\omega \in \mathcal{BP}^\Delta(\gamma)$ if and only if $\omega \in \mathcal{BP}^\Delta(C)$.
	\end{corollary}
	
	Another easy corollary we obtain is that splendid Rickard equivalences induce local normalizer splendid equivalences, again an analogue of a result of Boltje--Perepelitsky. 
	
	\begin{corollary}{c.f. \cite[Theorem 11.4]{BoPe20}} \label{normalizerequivalence}
		Let $C$ be a splendid Rickard equivalence between blocks $A\subseteq kG$ and $B\subseteq kH$ and let $(\Delta(P,\phi,Q), e\otimes f^*)$ be a $C$-Brauer pair. Set $I=N_G(P,e)$ and $J=N_H(Q,f)$ and let $\hat{e}$ and $\hat{f}$ denote the unique block idempotents of $k[N_G(P)]$ and $k[N_H(Q)]$ which cover $e$ and $f$, respectively. Then, 
		\[\hat{e}k[N_G(P)]e \otimes_{kI} \Ind^{I\times J}_{N_{I \times J}(\Delta(P,\phi, Q))}\big(eC(\Delta(P, \phi, Q))f\big) \otimes_{kJ} fk[N_H(Q)]\hat{f} \]
		is a splendid Rickard equivalence for blocks algebras $k[N_G(P)]\hat{e}$ and $k[N_H(Q)]\hat{f}$. 
	\end{corollary}
	\begin{proof}
		Given any Brauer pair $(P,e)$ of $kG$, we have a Morita equivalence ${}_{k[N_G(P)]\hat{e}}\catmod \to {}_{kIe}\catmod$, where $\hat{e}$ is the unique block of $kN_G(P)$ covering $e$. Explicitly, $\hat{e} = \tr^{N_G(P)}_{I}(e)$, and the Morita equivalence is induced by the $(kIe, k[N_G(P)]\hat{e})$-bimodule $e[N_G(P)]\hat{e}$ (see \cite{NT}). The result now follows since
        \[\Ind^{I \times J}_{N_{I \times J}(\Delta(P,\phi, Q))}\big(eC(\Delta(P, \phi, Q))f\big)\] is a splendid equivalence for $kIe$ and $kJf$.
	\end{proof}
	
	\section{Structural properties of splendid Rickard equivalences}
	
	In this section, we deduce an analogous result to the following theorem of Boltje and Perepelitsky.

    \begin{theorem}{\cite[Theorem 14.1, 14.3]{BoPe20}}
        Let $\gamma$ be a $p$-permutation equivalence between blocks $A$ and $B$ of $kG$ and $kH$, respectively. Let $N$ be an indecomposable $(A,B)$-bimodule appearing in $\gamma$. Then $\calB\calP(N) \subseteq \calB\calP(\gamma)$. Moreover, there is a unique indecomposable $(A,B)$-bimodule $M$ appearing in $\gamma$ with multiplicity $\pm 1$ whose vertex is of the form $\Delta(D,\phi, E)$ where $D$ is a defect group of $A$. If $M \not\cong N$, then $\calB\calP(N) \subset \calB\calP(M)$.
    \end{theorem}

    The module $M$ as above is called the \textit{maximal module} of $\gamma$. We begin proving the analogous statement with a block-wise reformulation of a theorem in \cite{SM23}.

    We first state an oft-used lemma. The proof of the following folklore lemma is from correspondence from Peter Webb.
    
    \begin{lemma}\label{webblemma}
        Let $C$ be a chain complex of $kG$-modules and suppose that for some $i\in \Z$, we have decompositions of $kG$-modules
        \begin{align*}
            C_i &= U_i \oplus V_i\\
            C_{i-1} &= U_{i-1} \oplus V_{i-1}
        \end{align*}
        Let $\eta_i: U_i \to C_i$ be inclusion and $\pi_{i-1}: C_{i-1} \to U_{i-1}$ be projection. Suppose that $\pi_{i-1}d_i\eta_i: U_i \to U_{i-1}$ is an isomorphism. Then there is a direct sum decomposition $C \cong D \oplus W$ as complexes, where \[W \cong 0 \to U_i \xrightarrow{\id} U_i \to 0 \] is contractible. 
    \end{lemma}
    \begin{proof}
        The map $d_i\eta_i: U_i \to C_{i-1}$ is injective since $\pi_{i-1}d_i\eta_i$ is, and we define \[W_{i-1} = d_i\eta_i(U_i)\cong U_i.\] Now $C_{i-1} = W_{i-1} \oplus V_{i-1}$ because $W_{i-1}$ projects isomorphically onto $U_{i-1}$. We see that \[0 \to U_i \xrightarrow{d_i\mid_{U_i}} W_{i-1} \to 0\] is a contractible subcomplex of $C$. This shows that we may assume 
        \begin{align*}
            C_i &= U_i \oplus V_i\\
            C_{i-1} &= U_{i-1} \oplus V_{i-1}
        \end{align*}
        where $0 \to U_i \xrightarrow{d_i\mid_{U_i}} U_{i-1}\to 0$ is a contractible subcomplex of $C$, adjusting the rest of the notation to fit the assumption. We define \[X_i = \{(-(d_i\mid_{U_i})\inv \pi^{i-1}d_i(x), x)\mid x\in V_i\}\subseteq C_i,\] and $C_i = U_i \oplus X_i$ because $X_i$ projects isomorphically to $V_i$. We calculate 
        \begin{align*}
            d_i(X_i) &= \{d_i(-d_i\mid_{U_i})\inv d_i(x) + d_i(x)\mid x \in V_i\}\\
            &= \{-\pi_{i-1}d_i(x) + d_i(x) \mid x\in V_i\}\\
            &\subseteq V_{i-1}
        \end{align*}
        Thus $C$ is the direct sum of the complexes \[0 \to U_i \xrightarrow{d_i\mid_{U_i}} U_{i-1} \to 0\] and \[\dots \to C_{i+1} \xrightarrow{\pi_{X_i}d_{i+1}}X_i\xrightarrow{d_i\mid_{X_i}} V_{i-1} \xrightarrow{d_{i-1}\mid_{V_{i-1}}} C_{i-2}\to \cdots.\]
    \end{proof}
	
	\begin{lemma}\label{brauerconstructionforbrauerpairsprops}Let $M$ and $N$ be $p$-permutation $kG$-modules and let $f:M\to N$ be a $kG$-module homomorphism. Assume that $N = N_1\oplus \cdots \oplus N_k$ is a decomposition into indecomposable summands. For each $1\leq i \leq k$, let $(P_i,e_i)$ be a maximal $N_i$-Brauer pair. Then $f$ is split surjective if and only if $e_i\cdot f(P_i)$ is surjective for all $1\leq i \leq k$.
    %----old statement-----%
		% \begin{enumerate}
		% 	\item Suppose that $M$ and $N$ are $p$-permutation modules. Then  $f$ is split injective (resp. split surjective) if and only if $e\cdot f(P)$ is injective (resp. surjective) for all $(P,e) \in \mathcal{BP}(kG)$.
		% 	\item Suppose that $N$ is a trivial source n indecomposable $kG$-module and $M$ is a $p$-permutation module. Then $f: M \to N$ is split surjective if and only if $e\cdot f(P)$ is surjective for all maximal $(P,e)\in \mathcal{BP}(N)$. 
		% 	\item If $f: M\to N$ is a $kG$-module homomorphism and $M, N$ are $p$-permutation modules, then $f$ is split surjective if and only if $e\cdot f(P)$ is surjective, where $(P,e)$ runs over all maximal Brauer pairs of indecomposable direct summands of $N$.
		% \end{enumerate}
	\end{lemma}
	\begin{proof}
        The forward implication is immediate. We induct on the number of indecomposable summands of $N$. For the base case, suppose that $N$ is indecomposable. Suppose that $(P,e)$ is a maximal $N$-Brauer pair and that $e\cdot f(P)$ is surjective. By Proposition~\ref{prop:M-Brauer pair properties}, $P$ is a vertex of $N$ and therefore $N(P)$ is indecomposable. Since $(P,e)$ is an $N$-Brauer pair, $eN(P) \neq 0$ which implies that $eN(P) = N(P)$. Since $e\cdot f(P)$ is surjective, it follows that $f(P)$ is surjective as well. By \cite[Proposition 5.10.7]{Li181}, $f$ is split surjective; we may use this proposition by passing to duals, as the internal hom $\Hom_{k}(-,k)$ is an antiequivalence of $K^b({}_{kG}\catmod)$. 

        For the inductive step, suppose $N = N_1 \oplus \cdots \oplus N_k$. Consider the projection of $f$ onto $N_1$, then by the base case, it follows that this is split surjective. Let $M = M_1 \oplus M'$, with the inclusion of $M_1 $ into $M$ composed with $f$ and projection onto $N_1$ an isomorphism, and $M'$ a complement of $M_1$. It follows by Lemma \ref{webblemma} that we may find a corresponding direct sum decomposition of $N = N_1 \oplus N'$, with $f = f_1 \oplus f'$ satisfying $f_1: M_1 \to N_1$ an isomorphism and $f': M' \to N'$ a homomorphism. Since $N'$ is $p$-permutation with one less indecomposable summand, the inductive hypothesis completes the proof. 
        
        %-----old proof-----%
		% Part (a) follows from the property of $p$-permutation modules that $f: M\to N$ is split injective (resp. split surjective) if and only if $f(P)$ is injective (resp. surjective) for all $P \in s_p(G)$, see \cite[Theorem 5.8.11]{Li181}, since any $kG$-module homomorphism is injective (resp. surjective) if and only if all of its block components are injective (resp. surjective). 
		
		% For (b), we have that $f$ is split surjective if and only if $f(P)$ is surjective. Moreover, $N(P)$ is a projective indecomposable $k[N_G(P)/P]$-module. If $f(P)$ is surjective, then $e\cdot f(P)$ is as well for the same reasoning as in (a). Conversely, if $e\cdot f(P)$ is surjective, then since $eN(P) = N(P)$, $f(P)$ is surjective as well, and (b) is shown.
		
		% Finally, (c) follows by an easy induction argument on the number of summands of $N$ which progresses down the poset of maximal $N$-Brauer pairs. 
	\end{proof}
	
	\begin{theorem}
		Let $C$ be a chain complex of $p$-permutation $kG$-modules, and let $\calX$ be a subset of $\calB\calP(kG)$ which is $G$-stable and closed (under $\leq$). The following are equivalent:
		\begin{enumerate}
			\item For all $(P,e) \not\in\mathcal{X}$, $e\cdot C(P)$ is acyclic.
			\item There exists a chain complex $D$ with $C\simeq D$ such that every trivial source module occurring up to isomorphism in $D$ has a maximal Brauer pair contained in $\calX$. 
		\end{enumerate} 
        If $C$ is indecomposable and satisfies (a), then $C$ itself has the property in (b).
	\end{theorem}
	\begin{proof}
		(b) implies (a) is straightforward. Let $n$ be the minimum integer for which $C_n \neq 0$. For each nonzero $C_n$, write $C_n = X_n \oplus Y_n$, where $X_n$ consists of all indecomposable summands which have a maximal pair contained in $\mathcal{X}$. Thus, $Y_n$ consists of all indecomposable summands of $C_n$ which have a maximal Brauer pair which is not contained in $\mathcal{X}$. We construct a chain complex $D$ identical to $C$, except in degree $n+1$, where we set $D_{n+1} = X_n \oplus C_{n+1}$, and add the identity map on $X_n$ to the differential of $D_{n+1}$, $d_{n+1}$. Notice that $f\cdot X_n(Q) = 0$ for all $(Q,f) \not\in \mathcal{X}$ since $\mathcal{X}$ is a $G$-stable ideal. Therefore, $f\cdot C(Q) \cong f\cdot D(Q)$. But since $f\cdot C(Q)$ is acyclic for all $(Q,f) \not \in \calX$ by assumption, so is $f\cdot D(Q)$. Therefore $d_{n+1}$ composed with projection onto $Y_n$ is split surjective by Lemma~\ref{brauerconstructionforbrauerpairsprops}, and by construction, it follows that $d_{n+1}$ is split surjective. Thus, we have that $D \cong D' \oplus (C_n \xrightarrow{\sim} C_n)$, where $D'$ has lowest nonzero degree $n+1$. 
		
		Note that $C$ is homotopy equivalent to a shift of the mapping cone of the chain complex homomorphism $D' \to X_n$. Moreover, if $(Q,f) \not\in \calX$, $f\cdot C(Q) \cong f\cdot D(Q) \simeq f\cdot D'(Q)$, and these complexes are by assumption acyclic. We now perform an inductive argument as follows. If the length of $C$ is one, then (b) holds, since there is only one nonzero term which must vanish after applying the Brauer construction at all $(Q,f) \not\in \calX$, and the rest follows by the previous lemma. 
  
        Now we assume (a) implies (b) holds for complexes of length at most $n$, and suppose $C$ has length $n+1$. If the length of $C$ is $n+1$, then $D'$ has length $n$. Since $f\cdot D'(Q)$ is also acyclic for all $(Q,f) \not\in \calX$, $D'$ contains only modules with maximal Brauer pair contained in $\calX$. Since $C$ is homotopy equivalent to a shift of the mapping cone $D' \to X_n$, and both $D'$ and $X_n$ (as complex) are homotopy equivalent to complexes with all degrees containing only trivial source modules with maximal Brauer pairs contained in $\calX$, we conclude $C$ also satisfies this property, by invariance of the mapping cone under homotopy equivalences. 
       
       Lastly, suppose that $C$ is indecomposable. Then there exists a contractible complex $C'$ such that $C\oplus C' \cong D$. Hence, $C$ is isomorphic to a direct summand of $D$ and therefore $C$ has the desired property.         
	\end{proof}
	
	\begin{corollary}\label{bpswhichappear}
		Let $C$ be an indecomposable splendid Rickard equivalence for block algebras $A\subseteq kG$ and $B\subseteq kH$ and let $(\Delta(D,\phi,E),e\otimes f^*)$ be a maximal $C$-Brauer pair. Every trivial source bimodule occurring in $C$ has a maximal Brauer pair contained in $(\Delta(D,\phi,E),e\otimes f^*)$. In particular, the set consisting of all Brauer pairs for every trivial source bimodule appearing in $C$ and the set of $C$-Brauer pairs coincide. 
	\end{corollary}
	\begin{proof}
		By Corollary~\ref{mainresult}, $\calB\calP(C)$ is a $G\times H$-stable ideal in $\calB\calP(k[G\times H])$ and any two maximal $C$-Brauer pairs are $G\times H$-conjugate. Thus, by the preceding theorem, every trivial source bimodule occurring in $C$ has a maximal Brauer pair contained in $(\Delta(D,\phi,E),e\otimes f^*)$. Denoting the set of all Brauer pairs which occur in any trivial source bimodule in $C$ by $\calX$, it follows that $\calX \subseteq \calB\calP(C)$. On the other hand, if $(\Delta(P,\phi, Q),e\otimes f^*)$ is a $C$-Brauer pair, then $eC(\Delta(P,\phi,Q))f \neq 0$. Thus, there exists a trivial source bimodule $M$ appearing in $C$ such that $(\Delta(P,\phi,Q),e\otimes f^*) \in \calB\calP(M)$. Thus, $\calB\calP(C) \subseteq \calX$. 
	\end{proof}
	
    By Proposition~\ref{prop:src indecomposable}, the preceding result applies to any splendid Rickard equivalence between blocks after removing contractible direct summands. Moreover, by using Theorem~\ref{thm:breaking up a splendid complex} one obtains an obvious analogue for a splendid Rickard equivalence between sums of blocks. Finally, we note that Corollary \ref{maximallocalbrauerpairs} also holds for $C$-Brauer pairs as well. 
    
	\textbf{Acknowledgements:} The authors would like to thank their advisor, Robert Boltje, for the many hours of guidance and support which made this paper possible, and would like to thank Peter Webb for communicating Lemma \ref{webblemma}.

	\bibliography{bib}

@article{BoXu07,
  author  = {R. Boltje and B. Xu},
  journal = {Trans. Amer. Math. Soc.},
  title   = {On $p$-permutation equivalences: betwen {R}ickard equivalences and isotypies},
  year={2008},
  volume={360},
  pages={5067-5087},
}

@Article{Ri96,
  author  = {J. Rickard},
  journal = {Proc. London Math. Soc.},
  title   = {Splendid equivalences: derived categories and permutation modules},
  year    = {1996},
  volume={72},
  pages={331-358},
}

@article{BoPe20,
  author = {R. Boltje and P. Perepelitsky},
  title  = {$p$-permutation equivalences between blocks of group algebras},
  journal  = {J. Algebra},
  year = 2025,
  pages = {815-887},
  volume = 664,
  numbwer = A
}

@Book{Li181,
  author    = {M. Linckelmann},
  editor    = {Ian J. Leary},
  publisher = {Cambridge University Press},
  title     = {The Block Theory of Finite Group Algebras, Volume 1},
  year      = {2018},
  address = {Cambridge}
}

@Book{Bo10,
  author    = {S. Bouc},
  publisher = {Springer-Verlag},
  title     = {Biset Functors for Finite Groups},
  year      = {2010},
  editor    = {J.-M. Morel, F. Takens, B. Teissier},
  address = {Berlin}
}

@article{Br85,
    author = {M. Brou\'e},
    title = {On {S}cott modules and $p$-permutation modules: An approach through the {B}rauer morphism},
    journal = {Proc. Amer. Math. Soc.},
    year = 1985,
    volume  = {93},
    number  = {3},
    pages   = {401-408}
}

@book{AKO11,
    author = {M. Aschbacher and R. Kessar and B. Oliver},
    title = {Fusion Systems in Algebra and Topology},
    publisher = {Cambridge University Press},
    year = {2011},
    address = {Cambridge}
}

@article{BrP80,
    author = {M. Brou\'e and L. Puig},
    title = {Charaters and local structure in $G$-algebras},
    journal = {J. Algebra},
    year = 1980,
    volume  = {63},
    pages   = {306-317}
}

@article{AB79,
    author = {J. Alperin and M. Brou\'e},
    title = {Local methods in block theory},
    journal = {Ann. Math.},
    year = 1979,
    volume  = {110},
    number  = {2},
    pages   = {143-157}
}

@article{L06,
    author = {M. Linckelmann},
    title = {Simple fusion systems and the {S}olomon 2-group},
    journal = {J. Algebra},
    year = 2006,
    volume  = {296},
    pages   = {385-401}
}

@article{Bo10b,
    author = {S. Bouc},
    title = {Bisets as categories and tensor product of induced bimodules},
    journal = {Appl. Categ. Struct.},
    year = 2010,
    volume  = {18},
    pages   = {517-521}
}

@article{BoDa12,
    author = {R. Boltje and S. Danz},
    title = {A ghost ring for the left-free double {B}urnside ring and an application to fusion systems},
    journal = {Adv. Math.},
    year = 2012,
    volume  = {229},
    pages   = {1688-1733}
}

@Conference{Br93,
  author    = {M. Brou\'e},
  booktitle = {International Conference on Group Theory ``Groups 93''},
  title     = {Rickard Equivalences and Block Theory},
  year      = {1993},
}

@article{SM23,
    author = {S. K. Miller},
    title = {Endotrivial complexes},
    journal = {J. Algebra},
    year = 2024,
    volume = 650,
    pages = {173-218}
}

@article{Li98,
    author = {M. Linckelmann},
    title = {On derived equivalences and local structure between blocks of finite groups},
    journal = {Turk. J. Math.},
    year = 1998,
    volume  = {22},
    number  = {1},
    pages   = {93-108}
}

@book{NT,
    author = {H. Nagao and Y. Tsushima},
    title = {Representations of Finite Groups},
    publisher = {Academic Press},
    year = 1989,
    address = {Boston}
}

@article{BG23a,
    author = {P. Balmer and M. Gallauer},
    title = {Finite permutation resolutions},
    journal = {Duke Math. J.},
    year = 2023,
    number = 2,
    pages = {201-229},
    volume = 172
}

@article{R91,
    author = {J. Rickard},
    title = {Derived equivalences as derived functors},
    journal = {J. London Math. Soc.},
    year = 1991,
    pages = {37-48},
    volume = 43
}

@book{Pu99,
    author = {L. Puig},
    title = {On the local structure of {M}orita and {R}ickard equivalences between Brauer blocks},
    publisher = {Birkh\"auser},
    year = 1999,
    address = {Basel},
    volume = 178,
}
    \bibliographystyle{plain}

\end{document}